\documentclass[12pt]{amsart}

\usepackage{amsxtra,amssymb,amsthm,amsmath,amscd,url,listings}
\usepackage[utf8]{inputenc}
\usepackage{eucal}
\usepackage{fullpage}
\usepackage{scrtime}
\usepackage[colorlinks]{hyperref}
\usepackage{xtab}

\usepackage{graphicx}

\newcounter{point}

\renewcommand{\leq}{\leqslant}
\renewcommand{\geq}{\geqslant}

\numberwithin{equation}{section}

\newcommand{\Cc}{\mathbf{C}}

\newcommand{\Zz}{\mathbf{Z}}

\newcommand{\Rr}{\mathbf{R}}

\newcommand{\Fp}{{\mathbf{F}_p}}

\newcommand{\Fpt}{{\mathbf{F}^\times_p}}

\newcommand{\proba}{\mathbf{P}}
\newcommand{\expect}{\mathbf{E}}
\newcommand{\mods}[1]{\,(\mathrm{mod}\,{#1})}

\DeclareMathOperator{\hypk}{Kl}
\DeclareMathOperator{\hypkm}{Kl_2^{(\cdot)}}

\newcommand{\kpath}{\mathrm{K}}

\newcommand{\st}{\mathrm{ST}}

\newcommand{\ra}{\rightarrow}

\DeclareMathOperator{\Imag}{Im}
\DeclareMathOperator{\Reel}{Re}

\DeclareMathOperator{\supp}{supp}

\newcommand{\eps}{\varepsilon}
\renewcommand{\rho}{\varrho}

\DeclareMathOperator{\SU}{SU}

\newcommand{\demi}{{\textstyle{\frac{1}{2}}}}

\newcommand{\fct}[1]{\mathcal{{#1}}}

\newcommand{\spt}{\mathcal{S}}

\DeclareMathSymbol{\gena}{\mathord}{letters}{"3C}
\DeclareMathSymbol{\genb}{\mathord}{letters}{"3E}

\theoremstyle{plain}
\newtheorem{theorem}{Theorem}[section]
\newtheorem{lemma}[theorem]{Lemma}

\newtheorem{corollary}[theorem]{Corollary}

\newtheorem{proposition}[theorem]{Proposition}
\newtheorem*{proposition*}{Proposition}

\theoremstyle{remark}

\theoremstyle{definition}

\newtheorem{example}[theorem]{Example}
\newtheorem{remark}[theorem]{Remark}

\begin{document}

\title{On the support of the Kloosterman paths}

\author{Emmanuel  Kowalski}
\address{ETH Z\"urich -- D-MATH\\
  R\"amistrasse 101\\
  8092 Z\"urich\\
  Switzerland} 
\email{kowalski@math.ethz.ch}

\author{Will Sawin}
\address{ETH Institute for Theoretical Studies, ETH Zurich, 8092 Zürich}
\email{william.sawin@math.ethz.ch}

\date{\today,\ \thistime}


\keywords{Exponential sums, Kloosterman sums, Kloosterman paths,
  support of random series, Fourier series}

\begin{abstract} 
  We obtain statistical results on the possible distribution of
  \emph{all} partial sums of a Kloosterman sum modulo a prime, by
  computing explicitly the support of the limiting random Fourier
  series of our earlier functional limit theorem for Kloosterman
  paths.
\end{abstract}

% \keywords{Kloosterman sums, Kloosterman sheaves, Riemann Hypothesis
%   over finite fields, random Fourier series, short exponential sums,
%   probability in Banach spaces}

\thanks{Supported partly by a DFG-SNF lead agency program grant (grant
  200021L\_153647). W.S. partially supported by Dr. Max
  R\"ossler, the Walter Haefner Foundation and the ETH Zurich
  Foundation.}

\maketitle

\section{Introduction}

Let $p$ be a prime number. For $(a,b)\in\Fpt\times \Fpt$, we denote
$$
\hypk_2(a,b;p)=\frac{1}{\sqrt{p}} \sum_{x\in\Fpt}
e\Bigl(\frac{ax+b\bar{x}}{p}\Bigr)
$$
(where $e(z)=e^{2i\pi z}$ for $z\in\Cc$) the normalized Kloosterman
sums modulo $p$.  As in our previous paper~\cite{ks}, we consider the
\emph{Kloosterman paths} $t\mapsto K_p(a,b)(t)$ for $0\leq t\leq 1$,
namely the random variables on the finite set $\Fpt\times\Fpt$
obtained by linearly interpolating the partial sums
$$
(a,b)\mapsto \frac{1}{\sqrt{p}} \sum_{1\leq x\leq j}
e\Bigl(\frac{ax+b\bar{x}}{p}\Bigr),\quad\quad 0\leq j\leq p-1
$$
that correspond to $t=j/(p-1)$ (see~\cite[\S 1]{ks}). The set
$\Fpt\times\Fpt$ is viewed as a probability space with the uniform
probability measure, denoted $\proba_p$.
\par
We proved~\cite[Th. 1.1, Th. 1.5]{ks} that as $p\to+\infty$, the
$C([0,1])$-valued random variables $K_p$ converge in law to the random
Fourier series
% and their limiting Fourier
% series, as defined and studied by Kowalski and Sawin~\cite{ks}, namely
$$
\kpath(t)%=\sum_{h\in\Zz} \frac{e(ht)-1}{2\pi ih}\st_h
=t\st_0+ \sum_{h\not=0}\frac{e(ht)-1}{2\pi ih}\st_h
$$
where $(\st_h)_{h\in\Zz}$ is a family of independent Sato-Tate random
variables (i.e., with law given by
$\displaystyle{\frac{1}{\pi}}\sqrt{1-\displaystyle{\frac{x^2}{4}}}\, dx$ on
$[-2,2$]) and the convergence holds almost surely in the sense of
uniform convergence of symmetric partial sums.
\par
We discuss in this paper the support of this random Fourier series
$\kpath(t)$, and the arithmetic consequences of its structure. We will
denote the support by $\spt$.

\begin{theorem}\label{th-support}
  The support $\spt$ of the law of $\kpath$ in $C([0,1])$ is the set
  of all $f\in C([0,1])$ such that $f(0)=0$, $f(1)\in [-2,2]$ and such
  that the function $g(t)=f(t)-tf(1)$ satisfies
  $\widehat{g}(h)\in i\Rr$ and
$$
|\widehat{g}(h)|\leq \frac{1}{\pi |h|}
$$
for all non-zero $h\in\Zz$, where
$$
\widehat{g}(h)=\int_0^1 g(t)e(-ht)dt
$$ 
are the Fourier coefficients of $g$.
\end{theorem}

See Section~\ref{sec-proofs} for the proof.  From the arithmetic point
of view, what matters is the combination of this result of the next
proposition.

\begin{proposition}\label{pr-arith}
  Let $f\in C([0,1])$ be a function in the support $\spt$ of $\kpath$.
  For any $\eps>0$, we have
\begin{multline*}
  \liminf_{p\to+\infty} \frac{1}{(p-1)^2}\Bigl|\Bigl\{ (a,b)\in
  \Fpt\times\Fpt\,\mid\,\\
  \max_{0\leq j\leq p-1} \Bigl| \frac{1}{\sqrt{p}} \sum_{1\leq x\leq
    j} e\Bigl(\frac{ax+b\bar{x}}{p}\Bigr)-f\Bigl(\frac{j}{p-1}\Bigr)
  \Bigr|<\eps \Bigr\}\Bigr|>0.
\end{multline*}
Conversely, if $\in C([0,1])$ does not belong to $\spt$, then there
exists $\delta>0$ such that
$$
\lim_{p\to+\infty} \frac{1}{(p-1)^2}\Bigl|\Bigl\{ (a,b)\in
\Fpt\times\Fpt\,\mid\,\\
\max_{0\leq j\leq p-1} \Bigl| \frac{1}{\sqrt{p}} \sum_{1\leq x\leq j}
e\Bigl(\frac{ax+b\bar{x}}{p}\Bigr)-f\Bigl(\frac{j}{p-1}\Bigr)
\Bigr|<\delta \Bigr\}\Bigr|=0.
$$
\end{proposition}

As an example, we obtain:

\begin{corollary}\label{cor-small}
For any $\eps>0$, we have
$$
\liminf_{p\to+\infty} \frac{1}{(p-1)^2}\Bigl|\Bigl\{ (a,b)\in
\Fpt\times\Fpt\,\mid\, \max_{0\leq j\leq p-1} \Bigl|
\frac{1}{\sqrt{p}} \sum_{1\leq x\leq j}
e\Bigl(\frac{ax+b\bar{x}}{p}\Bigr) \Bigr|<\eps \Bigr\}\Bigr|>0.
$$
\end{corollary}

Our goal, after proving these results, will be to illustrate them. We
begin in Section~\ref{sec-structure} by spelling out some properties
of the support of $\kpath$, some of which can be interpreted as
``hidden symmetries'' of the Kloosterman paths. Then we discuss some
concrete examples that we find interesting, especially various
polygonal paths in Section~\ref{sec-polygons}. In
Section~\ref{sec-fourier}, we consider functions \emph{not} in $\spt$
which can be brought to $\spt$ by change of variable.  We can show:

\begin{proposition}\label{pr-repar}
  Let $f\in C([0,1])$ be a \emph{real-valued} function such that
  $f(t)+f(1-t)=f(1)$ for all $t\in [0,1]$ and $|f(1)|\leq 2$.  Then
  there exists an increasing homeomorphism
  $\varphi\colon [0,1]\to [0,1]$ such that $\varphi(1-t)=1-\varphi(t)$
  for all $t$ and
% such that $\varphi(1-t)=1-\varphi(t)$ for all $t$, and
  $f\circ \varphi\in\spt$.
\end{proposition}

We will see that this is related to some classical problems of Fourier
analysis around the Bohr-P\'al Theorem.

We also highlight two questions for which we do not know the answer at
this time, and one interesting analogue problem:
\begin{enumerate}
\item Is there a space-filling curve in the support $\spt$ of
  $\kpath$?
\item Does Proposition~\ref{pr-repar} hold for complex-valued
  functions $f$ with $f(t)+\overline{f(1-t)}=f(1)$? (A positive
  answer would also give a positive answer to (1)).
\item What can be said about the support of the paths of partial
  \emph{character sums} (as in, e.g., the paper~\cite{bggk} of Bober,
  Goldmakher, Granville and Koukoulopoulos)?
\end{enumerate}
\par
\medskip
\par
\textbf{Acknowledgments}. The computations were performed using
\textsc{Pari/GP}~\cite{gp} and \textsc{Julia}~\cite{julia}; the plots
were produced using the \textsc{Gadfly.jl} package.
\par
\medskip
\par
\subsection*{Notation.} 
We denote by $|X|$ the cardinality of a set.  If $X$ is any set and
$f\colon X\to\Cc$ any function, we write (synonymously) $f\ll g$ for
$x\in X$, or $f=O(g)$ for $x\in X$, if there exists a constant
$C\geq 0$ such that $|f(x)|\leq Cg(x)$ for all $x\in X$. The ``implied
constant'' is any admissible value of $C$. It may depend on the set
$X$ which is always specified or clear in context.
% We write $f\asymp g$ if $f\ll g$ and $g\ll f$ are both true.
\par
We denote by $C([0,1])$ the space of all
continuous complex-valued functions on $[0,1]$.
\par
For any probability space $(\Omega,\Sigma,\proba)$, we denote by
$\proba(A)$ the probability of some event $A$, and for a $\Cc$-valued
random variable $X$ defined on $\Omega$, we denote by $\expect(X)$ the
expectation when it exists. We sometimes use different probability
spaces, but often keep the same notation for all expectations and
probabilities.

\section{Computation of the support}
\label{sec-proofs}

We begin with the proof of Theorem~\ref{th-support}. This uses a
standard probabilistic lemma, for which we include a proof for
completeness.

\begin{lemma}\label{lm-support}
  Let $B$ be a separable real or complex Banach space.  Let
  $(X_n)_{n\geq 1}$ be a sequence of \emph{independent} $B$-valued
  random variables such that the series $X=\sum X_n$ converges almost
  surely.
  % \footnote{\ By a result of P. L\'evy, this is equivalent in that
  % case to
  % convergence in law.}
  The support of the law of $X$ is the closure of the set of all
  convergent series of the form $\sum x_n$, where $x_n$ belongs to the
  support of the law of $X_n$ for all $n\geq 1$.
\end{lemma}

\begin{proof}
  For $N\geq 1$, we write
$$
S_N=\sum_{n=1}^N X_n,\quad\quad R_N=X-S_N.
$$
The variables $S_N$ and $R_N$ are independent. It is elementary (by
composition of the random vector $(X_1,\ldots,X_N)$ with the
continuous addition map) that the support of $S_N$ is the closure of
the set of elements $x_1+\cdots+x_N$ with $x_n\in\supp(X_n)$ for
$1\leq n\leq N$.
\par
We will prove that all convergent series $\sum x_n$ with
$x_n\in\supp(X_n)$ belong to the support of $X$, hence the closure of
this set is contained in the support of $X$.  Thus let $x=\sum x_n$ be
of this type. Let $\eps>0$ be fixed.
\par
For all $N$ large enough, we have 
$$
\Bigl\|\sum_{n>N}x_n\Bigr\|<\eps,
$$
and it follows that $x_1+\cdots +x_N$ belongs to the intersection of
the support of $S_N$ (by the previous remark) and of the open ball
$U_{\eps}$ of radius $\eps$ around $x$. Hence
$$
\proba(S_N\in U_{\eps})>0
$$
for all $N$ large enough. 
% ($U_{\eps}$ is an open neighborhood of some element in the support
% of $S_N$).
\par
Now the almost sure convergence implies (by the dominated convergence
theorem, for instance) that $\proba(\|R_N\|>\eps)\ra 0$ as $N\ra
+\infty$. Therefore, taking $N$ suitably large, we get
\begin{align*}
\proba(\|X-x\|<2\eps)&\geq \proba(\|S_N-x\|<\eps\text{ and }
\|R_N\|<\eps)\\
&= \proba(\|S_N-x\|<\eps)\proba(\|R_N\|<\eps)>0
\end{align*}
(by independence).  Since $\eps$ is arbitrary, this shows that
$x\in \supp(X)$, as was to be proved.
% \par
% Conversely, let $x\in\supp(X)$. For any $\eps>0$, we have
% $$
% \proba\Bigl(\Bigl\|\sum_{n\geq 1}X_n-x\Bigr\|<\eps\Bigr)>0.
% $$
% Since, for any $n_0\geq 1$, we have
% $$
% \proba\Bigl(\Bigl\|\sum_{n\geq 1}X_n-x\Bigr\|<\eps\text{ and }
% X_{n_0}\notin\supp(X_{n_0})\Bigr)=0,
% $$
% this means in fact that
% $$
% \proba\Bigl(\Bigl\|\sum_{n\geq 1}X_n-x\Bigr\|<\eps \text{ and }
% X_n\in\supp(X_n)\text{ for all } n \Bigr)>0.
% $$
% In particular, we can find $x_n\in\supp(X_n)$ such that  the series
% $\sum x_n$ converges and
% $$
% \Bigl\|\sum_{n\geq 1}x_n-x\Bigr\|<\eps,
% $$
% and hence $x$ belongs to the closure of the set of convergent series
% $\sum x_n$ with $x_n$ in the support of $X_n$ for all $n$.
\par
The converse inclusion (which we do not need anyway) is elementary
since for any $n$, we have $\proba(X_n\notin\supp(X_n))=0$.
\end{proof}

This almost immediately proves Theorem~\ref{th-support}, but some care
is needed since not all continuous periodic functions are the sum of
their Fourier series in $C([0,1]$).

\begin{proof}[Proof of Theorem~\ref{th-support}]
  Denote by $\widetilde{\spt}$ the set described in the statement. Then
  $\widetilde{\spt}$ is closed in $C([0,1])$, since it is the
  intersection of closed sets. Almost surely, a sample function
  $f\in C([0,1])$ of the random process $\kpath$ is given by a
  uniformly convergent series
$$
f(t)=\alpha_0t+ \sum_{h\not=0}\frac{e(ht)-1}{2\pi ih}\alpha_h
$$
(in the sense of symmetric partial sums) for some real numbers
$\alpha_h$ such that $|\alpha_h|\leq 2$ (\cite[Th. 1.1 (1)]{ks}). The
uniform convergence implies
$$
\widehat{g}(h)=\frac{\alpha_h}{2i\pi h},\quad\text{ where }\quad
g(t)=f(t)-tf(1),
$$
for $h\not=0$. Hence the function $f$ belongs to
$\widetilde{\spt}$. Consequently, the support of $\kpath$ is contained in
$\widetilde{\spt}$.
\par
We now prove the converse inclusion.  By Lemma~\ref{lm-support}, the
support $\spt$ contains the set of continuous functions with uniformly
convergent (symmetric) expansions
$$
t\alpha_0+ \sum_{h\not=0}\frac{e(ht)-1}{2\pi ih}\alpha_h
$$
where $\alpha_h\in [-2,2]$ for all $h\in\Zz$.  In particular, since
$0$ belongs to the support of the Sato-Tate measure, $\spt$ contains
all finite sums of this type.
\par
Let $f\in \widetilde{\spt}$. We have 
$$
g(t)=f(t)-tf(1)=\lim_{N\to +\infty} \sum_{|h|\leq N}
\widehat{g}(h)e(ht)\Bigl(1-\frac{|h|}{N}\Bigr),
$$
in $C([0,1])$, by the uniform convergence of Cesàro means of the
Fourier series of a continuous periodic function. Evaluating at $0$,
where $g(0)=0$, and subtracting yields
\begin{align*}
  f(t)&=tf(1)+\lim_{N\to +\infty} \sum_{|h|\leq N}
        \widehat{g}(h)(e(ht)-1)\Bigl(1-\frac{|h|}{N}\Bigr)\\
      &= tf(1)+\lim_{N\to
        +\infty} \sum_{|h|\leq N} \frac{\alpha_h}{2i\pi
        h}(e(ht)-1)\Bigl(1-\frac{|h|}{N}\Bigr)
\end{align*}
in $C([0,1])$, where $\alpha_h=2i\pi h\widehat{g}(h)$ for
$h\not=0$. Then $\alpha_h\in\Rr$ and $|\alpha_h|\leq 2$ by the
assumption that $f\in\widetilde{\spt}$, so each function
$$
tf(1)+\sum_{1\leq|h|\leq N}\frac{e(ht)-1}{2\pi
  ih}\alpha_h\Bigl(1-\frac{|h|}{N}\Bigr),
$$
belongs to $\spt$, by the result we recalled. Since $\spt$ is closed,
we conclude that $f$ also belongs to $\spt$.
\end{proof}

% \begin{corollary}
% Let $\eps>0$. We have
% $$
% \liminf_{p\to+\infty}
% \frac{1}{(p-1)^2}\Bigl|\Bigl\{
% (a,b)\in \Fpt\times\Fpt\,\mid\,
% \max_{0\leq j\leq p-1}
% \Bigl|
% \frac{1}{\sqrt{p}}
% \sum_{1\leq x\leq j}
% e\Bigl(\frac{ax+b\bar{x}}{p}\Bigr)
% \Bigr|<\eps
% \Bigr\}\Bigr|>0.
% $$
% \end{corollary}

% \begin{proof}
% Indeed the function $0$ belongs to the support of $\kpath(t)$.
% \end{proof}

We now prove the arithmetic statement of Proposition~\ref{pr-arith}.

\begin{proof}[Proof of Proposition~\ref{pr-arith}]
  Assume $f\in\spt$.  Since the $C([0,1])$-valued random variables
  $K_p$ converge in law to $\kpath$ as $p\to +\infty$
  (\cite[Th. 1.5]{ks}), a standard equivalent form of convergence in
  law implies that for any open set $U\subset C([0,1])$, we have
$$
\liminf_{p\to +\infty}\proba_p(K_p\in U)\geq \proba(\kpath\in U)
$$
(see~\cite[Th. 2.1, (i) and (iv)]{billingsley}).  If $f\in \spt$ and $U$
is an open neighborhood of $f$ in $C([0,1])$, then by definition we
have $\proba(\kpath\in U)>0$, and therefore
$$
\liminf_{p\to +\infty}\proba_p(K_p\in U)\geq \proba(\kpath\in U)>0.
$$
Take for $U$ the open ball of radius $\eps>0$ around $f$ so that
$K_p\in U$ if and only if
$$
\sup_{t\in [0,1]}|K_p(t)-f(t)|<\eps.
$$
Sampling the supremum at the points $t_j=j/(p-1)$ for
$0\leq j\leq p-1$, we deduce
$$
\liminf_{p\to+\infty}\proba_p\Bigl(\Bigl|K_p\Bigl(\frac{j}{p-1}\Bigr)-
f\Bigl(\frac{j}{p-1}\Bigr)\Bigr|<\eps\Bigr)>0,
$$
which translates exactly to the first statement.
\par
Conversely, if $f\notin \spt$, there exists a neighborhood $U$ of $f$
such that $\proba(\kpath\in U)=0$. For some $\delta>0$, this
neighborhood contains the closed ball $C$ of radius $\delta$ around
$f$, and by~\cite[Th. 2.1., (i) and (iii)]{billingsley}, we have
$$
0\leq \limsup_{p\to+\infty}\proba_p(K_p\in C)\leq \proba(\kpath\in
C)=0,
$$
hence the second assertion.
\end{proof}

\section{Structure and symmetries of the support}\label{sec-structure}

% This is most natural for arithmetic (see
% Corollary~\ref{cor-arith} below), but we will also consider the
% geometric shapes, where the freedom of reparameterizing a path leads
% to other questions.

% The functions in the support are therefore built out of functions
% $\varphi\colon [0,1]\to \Cc$ such that $\varphi(0)=\varphi(1)=0$
% % which take a purely imaginary value at $t=1/2$, 
% and whose graph is symmetric with respect to the imaginary axis; to
% these functions we add a ``stretch'' factor $\alpha_0t$.

% \par
% For such a function $f$, we have $\Reel(f(\demi))=\demi f(1)$,
% $$
% \int_0^1f(t)dt\in i\Rr,
% $$
% and the function $\varphi$ extended to a $1$-periodic function on
% $\Rr$ satisfies $\varphi(-t)=-\overline{\varphi(t)}$.

% Let $\fct{F}\subset C([0,1])$ denote the \emph{real} Banach space of
% all complex-valued continuous functions $f$ on $[0,1]$ such that
% \begin{equation}\label{eq-symmetry}
% f(t)-\Reel(f(\demi))=-\overline{f(1-t)}+\Reel(f(\demi))
% \end{equation}
% for $0\leq t\leq 1$ (this implies, in particular, that the image of
% $f$ is a subset of $\Cc$ that is symmetric with respect to the line
% $\Reel(z)=\demi \Reel(f(1))$ in $\Cc$).
% \par
% Let $\fct{F}_0$ denote the real subspace of $f\in\fct{F}$ such
% that $f(0)=0$. For $f\in \fct{F}_0$, we have $f(1)\in\Rr$.

We denote by $u$ the continuous linear map $C([0,1])\to C([0,1])$ such
that
$$
u(f)(t)=f(t)-f(1)t
$$
for all $t\in [0,1]$.
\par
Let $\fct{F}_0\subset C([0,1])$ denote the \emph{real} Banach space of
all complex-valued continuous functions on $[0,1]$ such that
\begin{equation}\label{eq-symmetry}
  f(t)+\overline{f(1-t)}=f(1)
\end{equation}
for all $t\in [0,1]$. This condition implies (taking $t=1/2$) that
$2\Reel(f(1/2))=f(1)$, hence in particular that $f(1)\in\Rr$. Taking
$t=0$, it follows also that $f(0)=0$. Writing the symmetry
relation~(\ref{eq-symmetry}) as
\begin{equation}\label{eq-sym-2}
  f(t)-\Reel(f(\demi))=-\overline{f(1-t)}+\Reel(f(\demi)),
\end{equation}
we see also that $\fct{F}_0$ is the subspace of functions satisfying
$f(0)=0$ among the space $\fct{F}$ of all complex-valued continuous
functions $f$ on $[0,1]$ that satisfy~(\ref{eq-sym-2}). This means, in
particular, that the image $f([0,1])\subset\Cc$ is symmetric with
respect to the line $\Reel(z)=\demi \Reel(f(1))$ in $\Cc$.
\par
The linear map $u$ induces by restriction an $\Rr$-linear map
$u\colon \fct{F}_0\to\fct{F}_0$.
%  such that
% $$
% u(f)(t)=f(t)-tf(1).
% $$
%% Well defined on F iff f(1) is real.
This is a continuous projection on $\fct{F}_0$ with $1$-dimensional
kernel spanned by the identity $t\mapsto t$, and with image the
subspace $\fct{F}_{1}\subset \fct{F}_0$ of functions such that
$f(1)=0$.

Theorem~\ref{th-support} implies that $\spt\subset \fct{F}_0$, where
the symmetry condition~(\ref{eq-symmetry}) follows from the fact that
the Fourier coefficients of the function $g=u(f)$ are purely
imaginary. More precisely, we have the following criterion that we
will use to check that concretely given functions in $\fct{F}_0$ are
in $\spt$:

\begin{lemma}\label{lm-expansion}
  Let $f\in C([0,1])$. Then $f\in \fct{F}_0$ if and only if there
  exist real numbers $\alpha_h$ for $h\in \Zz$ such that
$$
f(t)=\alpha_0t+ \lim_{N\to+\infty} \sum_{1\leq |h|\leq N}\alpha_h
\frac{e(ht)-1}{2i\pi h}\Bigl(1-\frac{|h|}{N}\Bigr).
$$
uniformly for $t\in [0,1]$. 
\par
For $f$ in $\fct{F}_0$, the expansion above holds if and only if
$$
\alpha_0=f(1),\quad\quad \alpha_h=f(1)+2i\pi h\widehat{f}(h)
\quad\text{ for } h\not=0.
$$
We have then $f\in \spt$ if and only if $|\alpha_h|\leq 2$ for all
$h\in\Zz$.
\end{lemma}

\begin{proof}
  This is a variant of part of the proof of Theorem~\ref{th-support}.
  The ``if'' statement follows from the uniform convergence by
  computation of the Fourier coefficients. For the ``only if''
  statement, consider any $f\in\fct{F}_0$, and write $u(f)$ as the
  uniform limit of its Cesàro means; evaluating at $t=0$ and using
  $u(f)(0)=0$, we obtain
$$
f(t)=\alpha_0t+ \lim_{N\to+\infty} \sum_{1\leq |h|\leq N}\alpha_h
\frac{e(ht)-1}{2i\pi h}\Bigl(1-\frac{|h|}{N}\Bigr).
$$
with $\alpha_0=f(1)$ and $\alpha_h=f(1)+2i\pi h\widehat{f}(h)$ for
$h\not=0$. The symmetry $f(t)+\overline{f(1-t)}=f(1)$ then shows that
$\alpha_h\in\Rr$. 
\par
The remaining statements are then elementary.
\end{proof}

The support $\spt$ has some symmetry properties that we now describe:
\par
\medskip
\par
\begin{enumerate}
\item The support $\spt$ of $\kpath$ is a subset of $\fct{F}_0$. It
  is closed, convex and balanced (i.e., if $f\in \spt$ and
  $\alpha\in [-1,1]$, then we have $\alpha f\in \spt$, see \cite[EVT,
  I, p. 6, déf. 3]{bourbaki}).  In particular, if $f$ is in $\spt$,
  then $-f$ is also in $\spt$.
\item We have $\bar{f}\in \spt$ if $f\in \spt$. In particular, we
  deduce that if $f\in \spt$, then $\Reel(f)=\demi (f+\bar{f})$ and
  $i\Imag(f)=\demi(f-\bar{f})$ are also in $\spt$; on the other hand,
  $\Imag(f)\in\spt$ only if $f$ is real-valued (so the imaginary is
  zero).
\item Denote by $\spt_1$ the intersection of $\spt$ and $\fct{F}_1$,
  i.e., those $f\in \spt$ with $f(1)=0$. Then $f\in \spt$ if and only
  if $u(f)\in \spt_1$ and $f(1)\in [-2,2]$.  In particular, we have a
  kind of ``action'' of $[-2,2]$ on $\spt$: given $f\in\spt$ and
  $\alpha\in\Rr$ such that $-2\leq\alpha+f(1)\leq 2$, the function
  given by $f_{\alpha}(t)=\alpha t+f(t)$ belongs to $\spt$ (and
  $f_{\alpha+\beta}(t)=(f_{\alpha})_{\beta}$, when this makes sense).
% For $f_0\in \spt_1$, all functions $f_{\alpha}$
%   defined for $\alpha\in [-2,2]$ by $f_{\alpha}(t)=\alpha t+f(t)$
%   belong to $\spt$ (and satisfy $u(f_{\alpha})=f_0$).
\item The support $\spt$ is ``stable under Fourier contractions'':
  given any subset $S$ of $\Zz$ and any $f\in\spt$, if a function
  $g\in\fct{F}_0$ satisfies $|g(1)|\leq |f(1)|$ and
  $|\widehat{u(g)}(h)|\leq |\widehat{u(f)}(h)|$ for alll $h\not=0$ in
  $\Zz$, then $g\in \spt$. 
\end{enumerate}

These are all immediate consequences of the description of
$\spt$. However, from the point of view of Kloosterman paths, they are
by no means obvious, and reflect hidden symmetry properties of the
``shapes'' of Kloosterman sums.
\par
\medskip
\par
The next remarks describe some ``obvious'' elements of $\spt$.

\begin{enumerate}
\item By a simple integration by parts, the support $\spt$ contains
  all functions $f$ such that $f(1)\in [-2,2]$ and $u(f)$ is in
  $C^1([0,1])\cap \fct{F}_0$ with $\|f'\|_{\infty}\leq 2$. More
  generally, it suffices that $u(f)$ be of total variation with the
  total variation of $u(f)$ at most $2$.
\item Let $g\colon [0,1]\to\Rr$ be a real-valued continuous function
  such that $g(0)=0$ and $g(1-t)=g(t)$ for all $t$. Then for any
  $\alpha$ with $|\alpha|\leq 2$, the function
$$
f(t)=\alpha t+ig(t)
$$
(whose image is, for $\alpha=1$, the graph of $f$) is in $\fct{F}_0$;
it belongs to $\spt$ if and only if the non-zero Fourier coefficients of
$g=u(f)/i$ satisfy
$$
|\widehat{g}(f)|\leq \frac{1}{\pi|h|}.
$$
\item Let $\fct{G}\subset \fct{F}_0$ be the real subspace of
  functions $f\in\mathcal{F}_0$ such that we have
$$
\|f\|_{\fct{G}}=\sup_{h\in\Zz}|h\widehat{u(f)}(h)|<+\infty,
$$
given the corresponding structure of Banach space (note that the only
constant function in $\fct{G}$ is the zero function to see that this
is a norm). This space contains all $C^1$ functions that belong to
$\fct{F}_0$ (in fact, it contains all functions $f$ of bounded
variation, and $\|f\|_{\fct{G}}$ is bounded by the total variation of
$f$ by~\cite[Th. II.4.12]{Z}).  We have $\spt\subset \fct{G}$, and
$\spt$ is the closed ball of radius $\pi^{-1}$ centered at $0$ in
$\fct{G}$. In particular, for any $f\in \fct{G}$, there exists
$\alpha>0$ such that $\alpha f\in \spt$. From the arithmetic point of
view, this means that any smooth enough curve satisfying the
``obvious'' symmetry condition can be approximated by Kloosterman
paths, after re-scaling it to bring the value at $1$ and the Fourier
coefficients in the right interval.
\end{enumerate}

% \begin{proof}
%   Let $f\in \spt$. As a continuous function, the function $u(f)$ has a
%   Fourier expansion
% $$
% u(f)(t)=\sum_{h\in\Zz}\widehat{u(f)}(h)e(ht)
% $$
% in $L^2([0,1])$. For $h\not=0$, we have
% $$
% \widehat{u(f)}(h)=\widehat{f}(h)-f(1)\int_0^1te(-ht)dt=
% \widehat{f}(h)+\frac{f(1)}{2i\pi h},
% $$
% and
% $$
% \int_0^1f(t)dt=\frac{f(1)}{2}+\int_0^1u(f)(t)dt=
% \frac{f(1)}{2}+\widehat{u(h)}(0).
% $$
% \end{proof}
\par
\medskip
\par
The support $\spt$ of $\kpath$ is, in any reasonable sense, a very
``small'' subset of the subspace $\fct{F}_0$ of $C([0,1])$. For
instance, the natural analogue of the Wiener measure on $\fct{F}_0$
is the series
$$
\mathrm{N}(t)=tN_0+ \sum_{h\not=0}\frac{e(ht)-1}{2\pi ih}N_h
$$
where $(N_h)$ are independent standard (real) gaussian random
variables. It is elementary that the support of $\mathrm{N}$ is
$\fct{F}_0$, whereas we have $\proba(N\in \spt)=0$.
\par
This sparsity property of $\spt$ means that the Kloosterman paths (as
parameterized paths) are rather special, and may explain why they seem
experimentally rather distinctive (at least to certain eyes). More
importantly maybe, this feature raises a number of interesting
questions that are simply irrelevant for Brownian motion or Wiener
measure: given some ``natural'' $f\in\fct{F}_0$, does it belong to
$\spt$ or not? This also contrasts with results like Bagchi's Theorem
for the functional distribution of (say) vertical translates of the
Riemann zeta function, where the support of the limiting distribution
is ``as large as possible'', given obvious restrictions
% , although this limiting distribution contains arithmetic
% information
(see~\cite{bagchi} and~\cite[\S 3.2, 3.3]{proba}; but see also
Remark~\ref{rm-bagchi} to see that there are interesting issues there
also).
\par
Another subtlety is that the question might be phrased in different
ways. A picture of a Kloosterman path, as in~\cite{ks}, only shows the
image $f([0,1])$ of a function $f\in\fct{F}_0$, and therefore
different functions lead to the same picture (we may replace $f$ by
$f\circ \varphi$ for any homeomorphism $\varphi\colon [0,1]\to [0,1]$
such that $\varphi(0)=0$ and $\varphi(1-t)=1-\varphi(t)$, which
implies that $f\circ \varphi$ is also in $\fct{F}_0$). So even if a
function $f\in\fct{F}_0$ does \emph{not} belong to $\spt$, we can ask
whether there exists a reparameterization $\varphi$ such that
$f\circ \varphi\in\spt$. Following this question leads to connections
with some classical problems of Fourier analysis, as we discuss in
Section~\ref{sec-fourier}.
\par
Finally, we remark that the support of $\kpath$ only depends on the
support of the Sato-Tate summand, and not on their particular
distribution. This implies that $\spt$ is also the support of similar
random Fourier series where the summands are independent and have
support $[-2,2]$. In particular, from the work of Ricotta and
Royer~\cite{rr}, this applies to the support of the random Fourier
series that appears as limit in law of the Kloosterman paths modulo
$p^n$ for fixed $n\geq 1$ and $p\to +\infty$, where the corresponding
Fourier series has summands $\mathrm{C}_h$ distributed like the trace
of a random matrix in the normalizer of the diagonal torus in
$\SU_2(\Cc)$. (Note however that the values of the liminf and limsup
in Proposition~\ref{pr-arith} do, of course, depend on the laws on the
summands).

\section{Elementary examples}

We present here a number of examples, in the spirit of curiosity.
Before we begin, we remark that since numerical inequalities are
important in determining whether a function $f\in C([0,1])$ belongs to
the support of $\kpath$, we have ``tested'' the following computations
by making, in each case, sample checks with \textsc{Pari/GP} to detect
multiplicative normalization errors.

\begin{example}
  Take $f(t)=\alpha t$ for some real number $\alpha$ with
  $|\alpha|\leq 2$. Then $f$ visibly belongs to the support of
  $\kpath(t)$ since $u(f)=0$. 
\par
In particular, for $\alpha=0$, we get Corollary~\ref{cor-small} from
Proposition~\ref{pr-arith}: for any $\eps>0$, we have
$$
\liminf_{p\to+\infty} \frac{1}{(p-1)^2}\Bigl|\Bigl\{ (a,b)\in
\Fpt\times\Fpt\,\mid\, \max_{0\leq j\leq p-1} \Bigl|
\frac{1}{\sqrt{p}} \sum_{1\leq x\leq j}
e\Bigl(\frac{ax+b\bar{x}}{p}\Bigr) \Bigr|<\eps \Bigr\}\Bigr|>0.
$$
So it is possible for the partial sums of the normalized Kloosterman
sum to remain at all time in an arbitrarily small neighborhood of the
origin.
\end{example}

\begin{example}
  Take $f(t)=i\alpha t(1-t)$ for some real number $\alpha$.  Then
  $f\in \fct{F}_0$. We compute (using Lemma~\ref{lm-expansion}) the
  coefficients $\alpha_h$ in the expansion
$$
f(t)=\alpha_0 t+ \lim_{N\to+\infty} \sum_{1\leq |h|\leq N}\alpha_h
\frac{e(ht)-1}{2i\pi h}\Bigl(1-\frac{|h|}{N}\Bigr),
$$
and find that $\alpha_0=0$ and $\alpha_h=\alpha(\pi h)^{-1}\in\Rr$ for
all $h\not=0$. In particular, we have $|\alpha_h|\leq 2$ for all $h$
if and only if $|\alpha|\leq 2\pi$.
\par
The graph of $f$ in that case is the vertical segment
$[0,i\alpha/4]$. So this parameterized segment $[0,iR]$ can be
approximated by the graph of a Kloosterman path as long as
$|R|\leq \pi/2$.
% Does this restriction also apply to other parameterizations? If so,
% this is a Buffon-needle for $\pi/2$: draw Kloosterman paths for
% large primes, and see how tall a path ``close'' to a vertical
% segment can be...
More precisely, Proposition~\ref{pr-arith} gives
\begin{multline*}
\liminf_{p\to+\infty} \frac{1}{(p-1)^2}\Bigl|\Bigl\{ (a,b)\in
\Fpt\times\Fpt\,\mid\,
\\ \max_{0\leq j\leq p-1} \Bigl|
\frac{1}{\sqrt{p}} \sum_{1\leq x\leq j}
e\Bigl(\frac{ax+b\bar{x}}{p}\Bigr)-
i\alpha\frac{j}{p-1}\Bigl(1-\frac{j}{p-1}\Bigr)
\Bigr|<\eps \Bigr\}\Bigr|>0,
\end{multline*}
if $|\alpha|\leq \pi/2$.
% \par
% Proposition~\ref{pr-repar} (that we will prove in
% Section~\ref{sec-fourier}) shows that there exists for \emph{any}
% $\alpha>0$ a reparameterization of $f$, i.e., an increasing
% homeomorphism $\varphi$ of $[0,1]$ such that
% $\varphi(1-t)=1-\varphi(t)$ for all $t$, with the property that
% $f\circ \varphi$ belongs to $\spt$. But, unless one knows what
% $\varphi$ is precisely, there is no clean arithmetic consequence of
% this.
\end{example}

\begin{example}
  Let $\alpha\in [-1,1]$ and consider the map
$$
f_1(t)=2\alpha t + i \sqrt{\alpha^2-\alpha^2(2t-1)^2},
$$
which parameterizes a semicircle above the real axis with diameter
$[0,2\alpha]$. The function $f_1$ belongs to $\fct{F}_0$.
\par
Let $\varphi_1=u(f_1)$. We have 
$$
\varphi_1(t)=i \sqrt{\alpha^2-\alpha^2(2t-1)^2},
$$
and using the computation of the Fourier transform of a semicircle
distribution (see, e.g.,~\cite[3.752 (2)]{gr}), we find
$$
\widehat{\varphi}_1(h)=\alpha(-1)^h\frac{ J_1(\pi h)}{2h},
$$
for $h\not=0$, where $J_1$ is the Bessel function of the first
kind. From Bessel's integral representation
$$
J_1(x)=\frac{1}{2\pi}\int_0^{2\pi} \cos(t-x\sin(t))dt,
$$
(see, e.g.,~\cite[p. 19]{watson}) we see immediately that
$|J_1(x)|\leq 1$ for all $x$ (in fact, the maximal value of the Bessel
function is about $0.58186$), hence the bound
$|\widehat{\varphi}_1(h)|\leq (\pi|h|)^{-1}$ holds for all $h\not=0$,
and therefore $f_1$ belongs to the support of $\kpath$ for
$|\alpha|\leq 1$.
\par
Now we consider a second parameterization of the same half circle,
namely 
$$ 
f_2(t)=2\alpha(1-\cos(\pi t)+i\sin(\pi t)),
$$
(more precisely, this is below the real axis if $\alpha<0$).  Let
$\varphi_2=u(f_2)$.  We compute
$$
\widehat{\varphi}_2(h)=2\alpha\Bigl(\frac{1}{i\pi
  h}-\frac{1}{i\pi(h+\demi)}\Bigr),
$$
from which it follows that $f_2$ also belongs to the support of
$\kpath$. 
\par
We see in particular here that the Kloosterman sum can follow this
semicircle in at least two ways...
\end{example}

\begin{example}
  For $t\in \Rr$, let $\langle t\rangle$ denote the distance to the
  nearest integer. The \emph{Takagi function} $\tau$ is the
  real-valued function defined on $[0,1]$ by
$$
\tau(t)=\sum_{j\geq 0}\frac{\langle 2^jt\rangle}{2^j}.
$$
It is continuous and nowhere differentiable, and has many remarkable
properties, including intricate self-similarity (see, e.g., the survey
by Lagarias~\cite{lagarias}). Since $\tau(1-t)=\tau(t)$ for
$t\in [0,1]$ and $\tau(1)=0$, the function $f$ giving the graph of
$\tau$, namely
$$
f(t)=t+i\tau(t),
$$
belongs to $\fct{F}_0$. Hata and Yamaguti computed the Fourier
coefficients of $\tau$, from which it follows that
$$
\widehat{u(f)}(h)=\frac{1}{2^mk^2i\pi^2}
$$
for $h\not=0$, when one writes $|h|=2^mk$ with $k$ an odd integer
(see, e.g.,~\cite[Th. 6.1]{lagarias}). Hence
$$
|\widehat{u(f)}(h)|\leq \frac{1}{2^mk\pi^2}\leq \frac{1}{\pi^2 |h|},
$$
and we can conclude that $f\in \spt$. An approximation of the graph of
$\tau$ is plotted in Figure~\ref{fig-t}.
%  (the image of $f=i\tau$ is of
% course the vertical segment $[0,i\|\tau\|_{\infty}]=[0,2i/3]$ in
% $\Cc$; see~\cite[Th. 3.2 (2)]{lagarias} for the maximum of $\tau$).

\begin{figure}
\caption{The Takagi function}\label{fig-t}
\includegraphics[height=6cm]{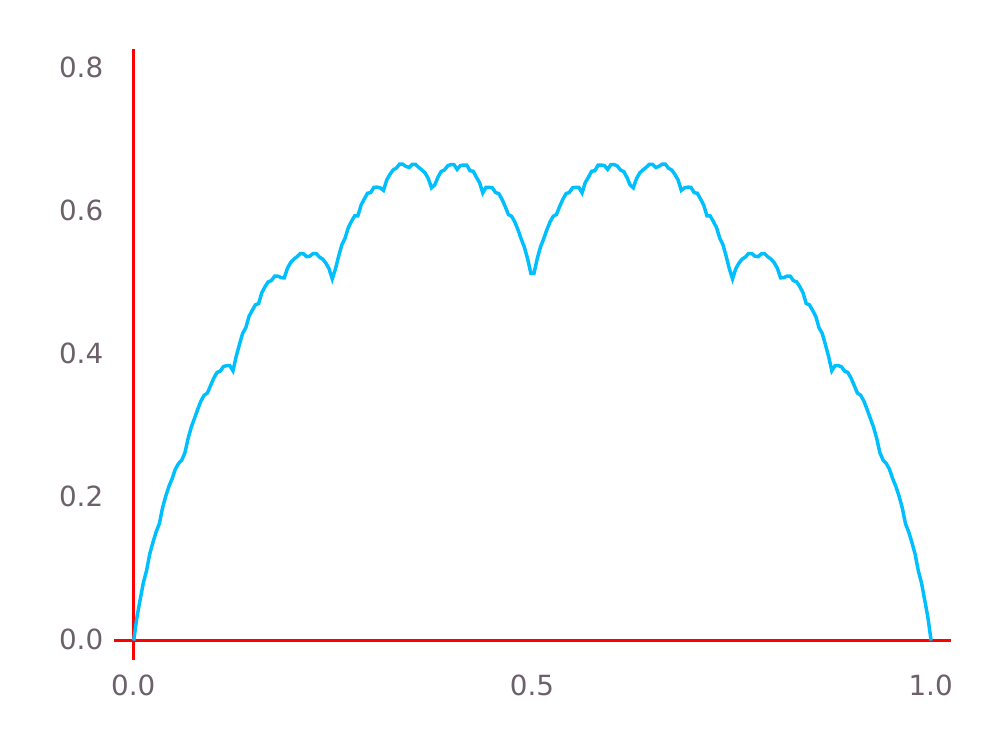}
\end{figure}

\end{example}

\begin{example}
Another famous function of real-analysis is Riemann's Fourier series
$$
\rho(t)=\sum_{n\geq 1}\frac{1}{\pi n^2}\sin(\pi n^2 t).
$$
This is a real-valued continuous $2$-periodic function such that
$\rho(0)=0$ and $\rho(t)+\rho(2-t)=0$ for all $t$. It is
non-differentiable \emph{except} at rational points $r=a/b$ with $a$
and $b$ coprime odd integers, where $\rho'(r)=-1/2$ (this is due to
Hardy for non-differentiability at irrational $t$, and to Gerver for
rational points; see Duistermaat's survey~\cite{duistermaat}, which
focuses on the links between $\rho$ and the classical theta function).
Define $f(t)=\rho(2t)$. Then $f$ is a real-valued element of
$\fct{F}_0$ with $u(f)=f$, and $\widehat{f}(h)=0$ if $|h|$ is not a
square, while
$$
\widehat{f}(\eps h^2)= \frac{\eps}{2i\pi h^2}
$$
for all $h\geq 1$ and $\eps\in\{-1,1\}$. Therefore $f\in\spt$. In
Figure~\ref{fig-rnd} is the \emph{graph} of $f$ (not the path
described by $f$, which is simply a segment of $\Rr$).

\begin{figure}
\caption{The Riemann function}\label{fig-rnd}
\includegraphics[height=6cm]{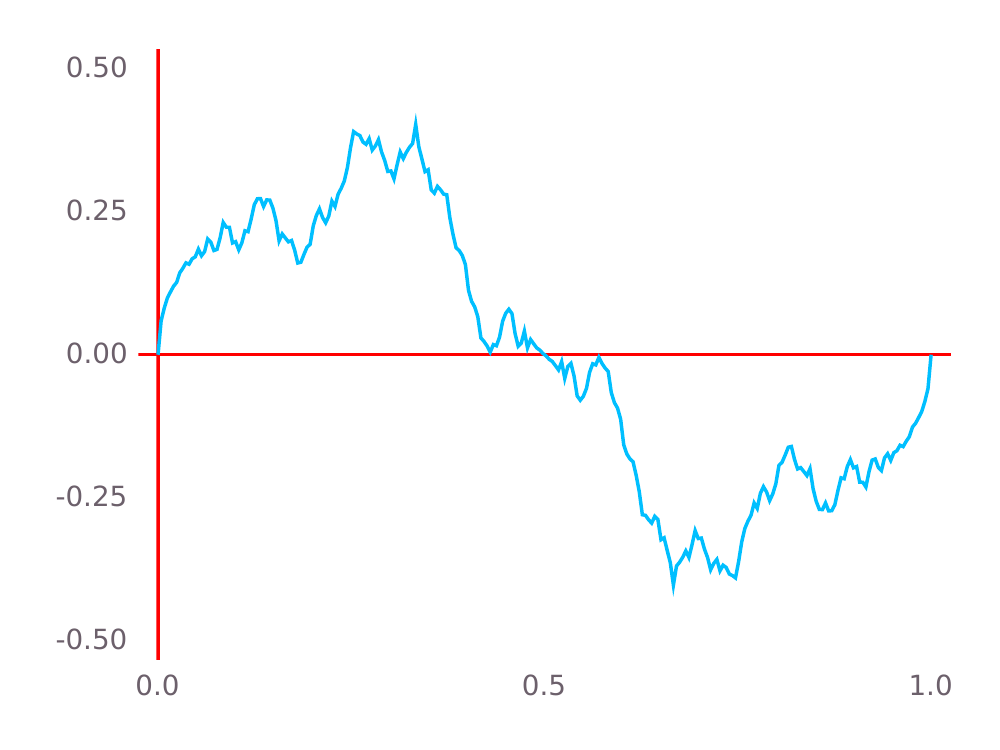}
\end{figure}

\end{example}

\begin{example}
  Yet another familiar example is the Cantor staircase function
  $\gamma$, which can be defined as
$\gamma(t)=\proba(X\leq t)$, where $X$ is the random series
$$
X=\sum_{k\geq 1} X_k
$$
with $(X_k)$ a sequence of independent random variables such that
$$
\proba(X_k=0)=\proba\Bigl(X_k=\frac{2}{3^k}\Bigr)=\frac{1}{2}
$$
for $k\geq 1$.
\par
The Cantor function satisfies $\gamma(0)=0$, $\gamma(1)=1$ and
$\gamma(t)+\gamma(1-t)=1$ for all $t$, hence $\gamma$ is a real-valued
element of $\fct{F}_0$. Computing using the probabilistic definition,
we obtain quickly the formula
$$
\widehat{u(\gamma)}(h)=\frac{(-1)^h}{2i\pi h} \prod_{k\geq
  1}\cos\Bigl(\frac{2\pi h}{3^k}\Bigr),
$$
from which we see that $\gamma\in\spt$.
\end{example}

\begin{example}
Let 
$$
f(t)=\sum_{h\geq 1}\mu(h)\frac{e(ht)-1}{2i\pi},
$$
where $\mu(h)$ denotes the Möbius function.  It is known (essentially
from work of Davenport, see~\cite{bateman-chowla} and~\cite[Th
13.6]{ant}, and from the Prime Number Theorem that implies that
$\sum \mu(h)h^{-1}=0$) that the series converges uniformly. Clearly
this function, which we call the Davenport function, belongs to
$\spt$. Its path is pictured in the left-hand graph of
Figure~\ref{fig-dav}.
\par
We may replace the Möbius function with the Liouville function, and we
also display the resulting path on the right-hand side of
Figure~\ref{fig-dav}.

\begin{figure}
\caption{The Davenport function and its variant}\label{fig-dav}
\includegraphics[height=6cm]{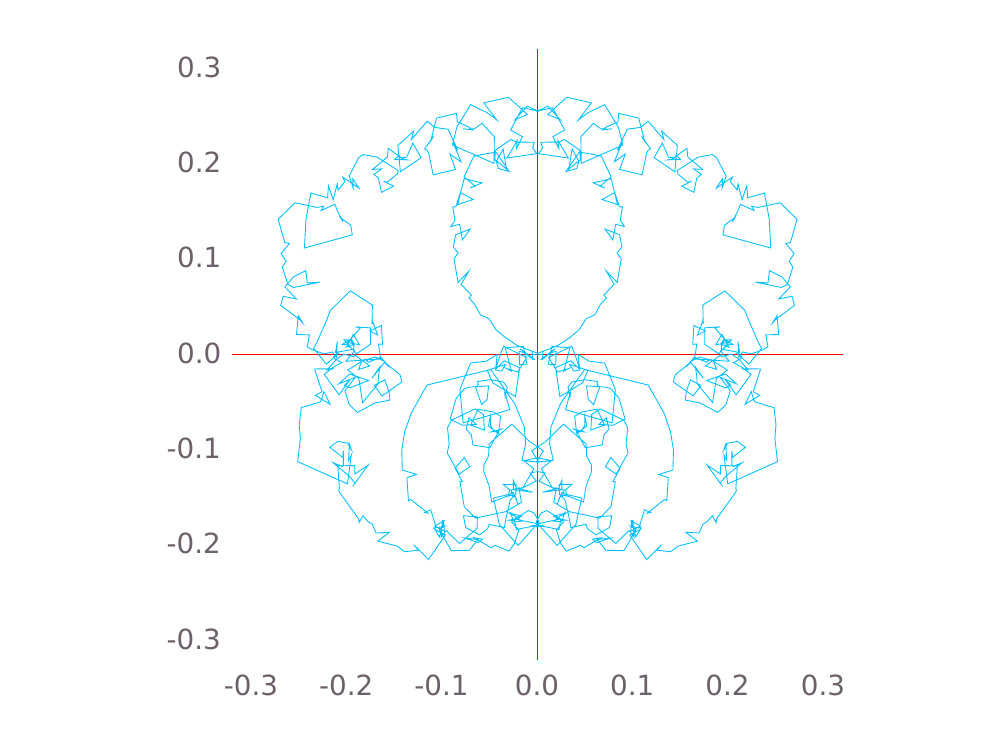}
\includegraphics[height=6cm]{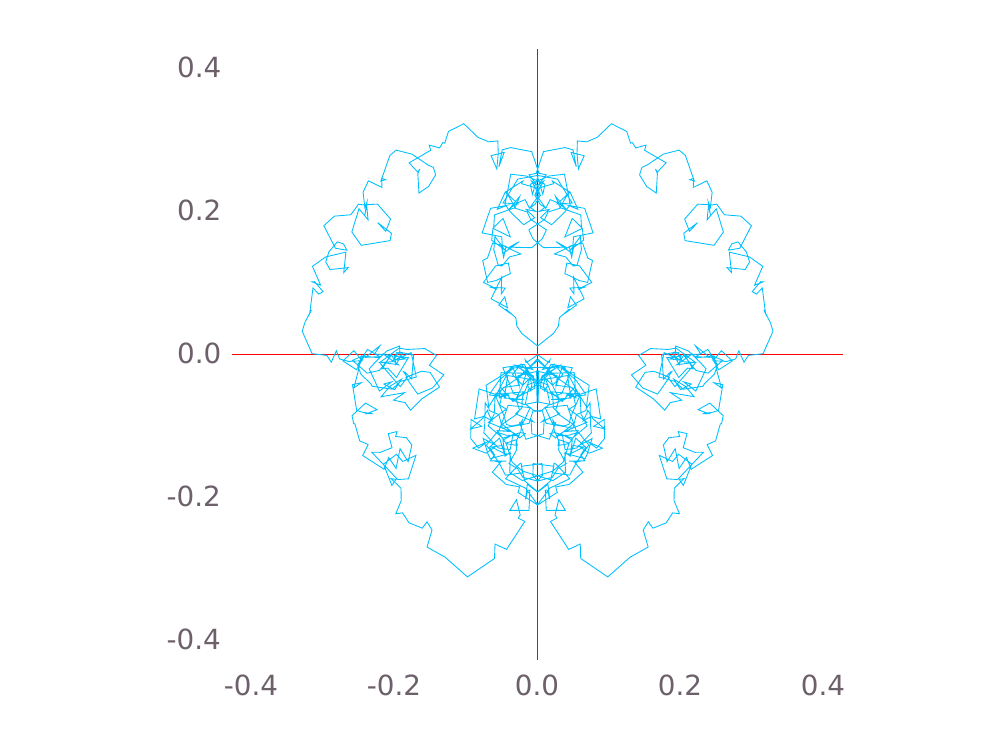}
\end{figure}

\end{example}

\section{Polygonal paths}\label{sec-polygons}

Polygonal paths provide a very natural class of examples of functions,
and we will consider a number of them. We begin with some elementary
preparation.
\par
Let $z_0$ and $z_1$ be complex numbers, and $t_0<t_1$ real numbers. We
define $\Delta=t_1-t_0$ and $f\in C([0,1])$ by
$$
f(t)=
\begin{cases}
\frac{1}{\Delta}(z_1(t-t_0)+z_0(t_1-t))&\text{ if } t_0\leq t\leq t_1,\\
0&\text{ otherwise,}
\end{cases}
$$
which parameterizes the segment from $z_0$ to $z_1$ during the
interval $[t_0,t_1]$.
\par
Let $h\not=0$ be an integer. By direct computation, we find
\begin{align*}
  \widehat{f}(h)
  &=-\frac{1}{2i\pi h} (z_1e(-ht_1)-z_0e(-ht_0))+
    \frac{1}{2i\pi
    h}(z_1-z_0)e(-ht_0)\frac{1}{\Delta}\Bigl(\int_0^{\Delta}e(-hu)du\Bigr)\\
  &=
    -\frac{1}{2i\pi h} (z_1e(-ht_1)-z_0e(-ht_0))+
    \frac{1}{2i\pi
    h}(z_1-z_0)
    \frac{\sin(\pi h\Delta)}{\pi h
    \Delta}e\Bigl(-h\Bigl(t_0+\frac{\Delta}{2}\Bigr)\Bigr).
\end{align*}
\par
Consider now an integer $n\geq 1$, a family $(z_0,\ldots,z_n)$ of
complex numbers and a family $(t_0,\ldots, t_n)$ of real numbers with
$$
0=t_0<t_1<\ldots<t_{n-1}<t_n=1.
$$
Let $f_j$ be the function as above relative to the points
$(z_{j},z_{j+1})$ and the interval $[t_j,t_{j+1}]$, and let 
function
$$
f=\sum_{j=0}^{n-1} f_j
$$
(in other words, $f$ parameterizes the polygonal path joining $z_0$ to
$z_1$ to \ldots\ to $z_n$, over intervals $[t_0,t_1]$, \ldots,
$[t_{n-1},t_n]$). Let $\Delta_{j}=t_{j+1}-t_{j}$.
\par
For $h\not=0$, we obtain by summing the previous expression, and using
a telescoping sum
\begin{equation}\label{eq-fourier-polygonal}
\widehat{f}(h)=-\frac{1}{2i\pi h} (z_n-z_0)+ \frac{1}{2i\pi h}
\sum_{j=0}^{n-1}(z_{j+1}-z_j)e\Bigl(-h\Bigl(t_j+\frac{\Delta_j}{2}\Bigr)
\Bigr)\frac{\sin(\pi h\Delta_j)}{\pi h \Delta_j}.
% ) \frac{1}{\Delta_{j}}
% \int_0^{\Delta_{j}}e(-hu)du.
\end{equation}
\par
Now assume further that $\Delta_j$ is constant for $0\leq j\leq n-1$,
equal to $1/n$. We then have $t_j=j/n$, and we obtain
\begin{equation}\label{eq-fourier-pol-2}
\widehat{f}(h)=-\frac{1}{2i\pi h} (z_n-z_0)+ \frac{1}{2i\pi h}
\frac{\sin(\pi h/n)}{\pi h/n}
\sum_{j=0}^{n-1}(z_{j+1}-z_j)e\Bigl(-\frac{h(j+\demi)}{n}\Bigr).
\end{equation}
% where 
% $$
% \beta_n=n \int_0^{1/n}e(-hu)du.
% $$
% We note that $|\beta_n|\leq 1$. 
It is elementary that $f$ belongs to $\fct{F}_0$ if and only if
$z_0=0$ and if the sums
\begin{equation}\label{eq-sums-polygonal}
  \widetilde{f}(h)=\sum_{j=0}^{n-1}(z_{j+1}-z_j)e\Bigl(-\frac{h(j+\demi)}{n}\Bigr)
\end{equation}
are real-valued.  If this is the case, then the polygonal function $f$
belongs to $\spt$ if and only if $|z_n|\leq 2$ and
\begin{equation}\label{eq-cond-polygonal}
  \Bigl|\frac{\sin(\pi h/n)}{\pi h/n}
  \widetilde{f}(h)\Bigr|=  \Bigl|\frac{\sin(\pi h/n)}{\pi h/n}
  \sum_{j=0}^{n-1}(z_{j+1}-z_j)e\Bigl(-\frac{hj}{n}\Bigr)
  \Bigr|\leq 2
\end{equation}
for all $h\not=0$ (disregarding the constant phase $e(-h/(2n))$,
although it is important to ensure that the exponential sums are
real-valued).

\begin{example}\label{ex-klpath1}
  The first polygonal paths that we consider are -- naturally enough
  -- the Kloosterman paths themselves.

  \emph{Fix} an odd prime $p$ and integers $a$ and $b$ coprime to
  $p$. Let $f\in C([0,1])$ be the function given by the Kloosterman
  path $K_p(a,b)$. It is an element of $\fct{F}_0$, and we can
  interpret it as a polygonal function with the following data:
  $n=p-1$, $t_j=j/(p-1)$ for $0\leq j\leq p$, and
$$
z_j=\frac{1}{\sqrt{p}}\sum_{1\leq x\leq
  j}e\Bigl(\frac{ax+b\bar{x}}{p}\Bigr),\quad\quad 0\leq j\leq p-1.
$$
Since $z_{p-1}$ is the normalized Kloosterman sum $\hypk_2(a,b;p)$, we
have $z_{p-1}\in [-2,2]$ by the Weil bound. Since
$$
z_{j+1}-z_j=\frac{1}{\sqrt{p}}
e\Bigl(\frac{a(j+1)+b\overline{(j+1)}}{p}\Bigr),
$$
the condition~(\ref{eq-cond-polygonal}) becomes 
$$
\Bigl|\frac{\sin(\pi h/(p-1))}{\pi h/(p-1)} \sum_{x=1}^{p-1}
e\Bigl(\frac{ax+b\bar{x}}{p}\Bigr) e\Bigl(-\frac{hx}{p-1}\Bigr)
\Bigr|\leq 2\sqrt{p}
$$
for all non-zero integers $h$ (after a change of variable), or indeed
for $1\leq h\leq p(p-1)$, by periodicity of $\widetilde{f}(h)$, since
the function $x\mapsto |\sin(\pi x/p)/(x/p)|$ is decreasing along
arithmetic progressions modulo $p(p-1)$.
\par
The inner sum is not quite the Kloosterman sum $\hypk_2(a-h,b;p)$, or
any other complete exponential sum. In particular, whether the desired
condition is satisfied is not obvious at all. It suffices that
\begin{equation}\label{eq-easy-cond}
\Bigl|\frac{1}{\sqrt{p}} \sum_{x=1}^{p-1}
e\Bigl(\frac{ax+b\bar{x}}{p}\Bigr) e\Bigl(-\frac{hx}{p-1}\Bigr)
\Bigr|\leq 2
\end{equation}
for $1\leq h\leq p(p-1)$ (by periodicity), but this is not a necessary
condition.
\par
We provide some numerical illustrations. In the following table, we
indicate for various primes $p$ how many $a\in \Fpt$ are such that the
Kloosterman path $K_p(a,1)$ modulo $p$ is in $\spt$, how many satisfy
the sufficient condition~(\ref{eq-easy-cond}) and how many are not in
$\spt$.

\begin{center}
\bottomcaption{Kloosterman paths $K_p(a,1)$}
\tablehead{
\hline
&
In $\spt$ with~(\ref{eq-easy-cond})
&
In $\spt$ without~(\ref{eq-easy-cond})
&
Not in $\spt$
\\
\hline
}
{\small
\renewcommand{\arraystretch}{1.3}
\begin{xtabular}{c|c|c|c}
  $5$ & $4$ & $0$ & $0$\\
  \hline
  $7$ & $6$ & $0$ & $0$\\
  \hline
  $13$ & $9$ & $3$ & $0$\\
  \hline
  $19$ & $1$ & $14$ & $3$\\
  \hline
  $23$ & $9$ & $13$ & $0$\\
  \hline
  $29$ & $28$ & $0$ & $0$\\
  \hline
  $229$ & $0$ & $133$ & $95$\\
  \hline
  $233$ & $0$ & $126$ & $106$\\
  \hline
  $541$ & $0$ & $0$ & $540$\\
  \hline
  $557$ & $0$& $27$ & $529$\\
  \hline
\end{xtabular}
}\end{center}

Maybe there are only finitely many Kloosterman paths in $\spt$?  The
``first'' example of a Kloosterman path not in $\spt$ is
$K_{19}(8,1)$. We picture it in Figure~\ref{fig-1} (and observe that
it looks a lot like a shadok). %%~\cite{shadoks}).

\begin{figure}
\caption{The Kloosterman path $K_{19}(8,1)$}\label{fig-1}
\includegraphics[height=6cm]{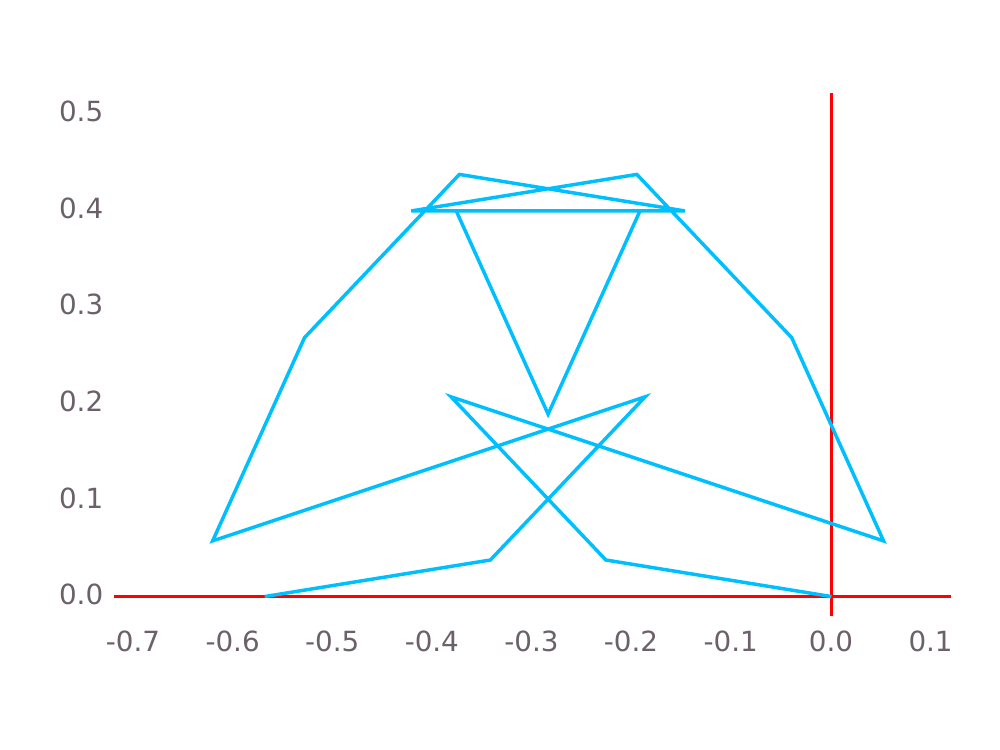}
\end{figure}

\begin{remark}\label{rm-bagchi}
  The analogue question for other probabilistic number theory results
  can also be of interest, and quite deep: if we consider Bagchi's
  results (\cite[Ch. 5]{bagchi}) concerning vertical translates of the
  Riemann zeta function restricted to a fixed small circle in the
  strip $1/2<\mathrm{Re}(s)<1$, then we see that the Riemann
  Hypothesis for the Riemann zeta function is equivalent to the
  statement that, for any $t\in\Rr$, and any such disc, the
  restriction of $s\mapsto \zeta(s+it)$ belongs to the support of the
  limiting distribution.
\end{remark}

% \begin{remark}
%   The computation above suggests an alternate approach to the main
%   results of~\cite{ks}. Namely, we view again $(a,b)\mapsto K_p(a,b)$
%   as a $C([0,1])$-valued random variable on $\Fpt\times\Fpt$,
%   by~(\ref{eq-fourier-pol-2}). For $h\not=0$, the $h$-th Fourier
%   coefficient of the $1$-periodic function
%   $t\mapsto K_p(a,b)-t\hypk_2(a,b;p)$ is the random variable mapping
%   $(a,b)$ to
% $$
% \frac{1}{2i\pi h}\frac{\sin(\pi h(p-1))}{\pi h/(p-1)}
% e\Bigl(-\frac{h}{2(p-1)}\Bigr)
% \sum_{x=1}^{p-1}e\Bigl(\frac{ax+b\bar{x}}{p}\Bigr)
% e\Bigl(-\frac{hx}{p-1}\Bigr).
% $$
% For fixed $h\not=0$, it follows elementarily from Katz's
% equidistribution theorem that this random variable converges to
% $(2i\pi h)^{-1}\st_h$, where $\st_h$ is a Sato-Tate random variable,
% and indeed (as in the crucial step in Lemma 2.5 of~\cite{ks}) that the
% family of all Fourier coefficients converges in law to the family
% $(\st_h)$ of independent Sato-Tate random variables. Combined with the
% tightness of the random variables $K_p$ (\cite[\S 3]{ks}), this
% implies after some re-arrangement the convergence in law of $K_p$ to
% $\kpath$. 
% \par
% Note that an argument of this type still uses the same crucial
% ingredients (Sato-Tate limiting distribution and independence of
% shifted Kloosterman sums, and tightness) as the proof in~\cite{ks}. It
% is also a bit less quantitative.
% \end{remark}

\end{example}

\begin{example}\label{ex-klpath2}
  We now consider a variant of Kloosterman paths (the Swiss railway
  clock version) where the partial sums are joined with intervals of
  length $1/p$, but a pause (of duration $1/p$) is inserted at the
  ``middle point'' (the second hand of a Swiss railway clock likewise
  stops about a second and a half at the beginning of each minute).
  %% https://www.eda.admin.ch/aboutswitzerland/en/home/dossiers/einleitung---schweizer-uhren/spezielle-uhren--schweizer-bahnhofsuhr.html
\par
This means that we consider again a fixed odd prime $p$ and
$(a,b)\in\Fpt\times\Fpt$, and the polygonal path with $n=p$, $t_j=j/p$
and
$$
z_j=\frac{1}{\sqrt{p}}\sum_{1\leq x\leq
  j}e\Bigl(\frac{ax+b\bar{x}}{p}\Bigr)\quad\text{ for } 0\leq j\leq
(p-1)/2,
$$
%% then $z_{(p+1)/2}=z_{(p-1)/2}$ (the pause) and finally
and
$$
z_j=\frac{1}{\sqrt{p}}\sum_{1\leq x\leq
  j-1}e\Bigl(\frac{ax+b\bar{x}}{p}\Bigr)\quad\text{ for } (p+1)/2\leq
j\leq p,
$$
which means in particular that $z_{(p-1)/2}=z_{(p+1)/2}$, representing
the pause. Because this pause comes in the middle of the path, we have
$f\in\fct{F}_0$.
\par
We get
$$
z_{j+1}-z_j=\frac{1}{\sqrt{p}}
e\Bigl(\frac{a(j+1)+b\overline{(j+1)}}{p}\Bigr),
$$
if $0\leq j\leq (p-3)/2$, $z_{(p+1)/2}-z_{(p-1)/2}=0$ and
$$
z_{j+1}-z_j=\frac{1}{\sqrt{p}}
e\Bigl(\frac{aj+b\bar{\jmath}}{p}\Bigr),
$$
if $(p+1)/2\leq j\leq p-1$. Hence the sums $\widetilde{f}(h)$ given
by~(\ref{eq-sums-polygonal}) become
\begin{multline*}
  \frac{1}{\sqrt{p}} \sum_{x=0}^{(p-3)/2}
  e\Bigl(\frac{a(x+1)+b\overline{(x+1)}}{p}\Bigr)
  e\Bigl(-\frac{h(x+\demi)}{p}\Bigr) +\frac{1}{\sqrt{p}}
  \sum_{x=(p+1)/2}^{p-1} e\Bigl(\frac{ax+b\bar{x}}{p}\Bigr)
  e\Bigl(-\frac{h(x+\demi)}{p}\Bigr)\\
  = \frac{1}{\sqrt{p}}e\Bigl(-\frac{h}{2p}\Bigr) \sum_{x=1}^{(p-1)/2}
  e\Bigl(\frac{(a-h)x+b\bar{x}}{p}\Bigr) + \frac{1}{\sqrt{p}}
  e\Bigl(\frac{h}{2p}\Bigr)\sum_{x=(p+1)/2}^{p-1}
  e\Bigl(\frac{(a-h)x+b\bar{x}}{p}\Bigr),
\end{multline*}
for all non-zero integers $h$ (it is more convenient here to keep the
phase).
\par
These are again close to the Kloosterman sums $\hypk_2(a-h,b;p)$, but
slightly different. Precisely, let
$$
\hypkm(a,b;p)=K_p(a,b)(1/2)=\frac{1}{\sqrt{p}}\sum_{x=1}^{(p-1)/2}
e\Bigl(\frac{ax+b\bar{x}}{p}\Bigr)
$$
denote the ``mezzo del cammin'' of the Kloosterman path, so that
$$
2\Reel(\hypkm(a,b;p))=\hypk_2(a,b;p).
$$
The sum $\widetilde{f}(h)$ above is then equal to
\begin{multline*}
  e\Bigl(-\frac{h}{2p}\Bigr)\hypkm(a-h,b;p) +
  e\Bigl(\frac{h}{2p}\Bigr)\overline{\hypkm(a-h,b;p)}=\\
  \cos(\pi h/p)\hypk_2(a-h,b;p)+
  2\sin(\pi h/p)\Imag(\hypkm(a-h,b;p)).
\end{multline*}
\par
To have $f\in\spt$ in this case, we must have
$$
\Bigl|\frac{\sin(\pi h/p)}{\pi h/p} \widetilde{f}(h) \Bigr|\leq 2
$$
for all $h\not=0$, or (by periodicity of $\widetilde{f}(h)$ and decay
of $x\mapsto |\sin(\pi x/p)/(x/p)|$ along arithmetic progressions
modulo $p$) when $1\leq h\leq p-1$.  Whether this holds or not depends
on the values of the imaginary part of $\hypkm(a-h,b;p)$ as $h$
varies. As in the previous example, it suffices that
\begin{equation}\label{eq-middle-simple}
|\widetilde{f}(h)|\leq 2
\end{equation}
for $1\leq h\leq p-1$.
\par
It follows from~\cite[Prop. 4.1]{ks} that when $p$ is large the random
variable $a\mapsto \Imag(\hypkm(a,b;p))$ on $\Fpt$ takes (rarely but
with positive probability) arbitrary large values. This indicates that
the property above becomes more difficult to achieve for large
$p$. Again, we present numerical illustrations.
% , which show a markedly
% different behavior from that of the Kloosterman paths themselves.

\begin{center}
\bottomcaption{Swiss Railway Clock Kloosterman paths $K_p(a,1)$}
\tablehead{
\hline
&
In $\spt$ with~(\ref{eq-middle-simple})
&
In $\spt$ without~(\ref{eq-middle-simple})
&
Not in $\spt$
\\
\hline
}
{\small
\renewcommand{\arraystretch}{1.3}
\begin{xtabular}{c|c|c|c|c}
  $5$ & $4$ & $0$ & $0$\\
  \hline
  % $7$ & $6$ & $0$\\
  % \hline
  $17$ & $14$ & $1$ & $1$\\
  \hline
  % $19$ & $17$ & $1$\\
  % \hline
  $23$ & $19$ & $2$ & $1$\\
  \hline
  $29$ & $26$ & $2$ & $0$\\
  \hline
  $229$ & $204$ & $17$ & $7$\\
  \hline
  $541$ & $484$ & $36$ & $20$\\
  \hline
  $1223$ & $1088$ & $94$ & $40$\\
  \hline
  $1987$ & $1763$ & $172$ & $51$\\
  \hline
  $2741$ & $2416$ & $239$ & $85$\\
  \hline
  $3571$ & $3176$ & $281$ & $113$\\
  \hline
\end{xtabular}
}\end{center}

The first case of a Swiss Clock Kloosterman path that is not in $\spt$
is the one corresponding to $K_{17}(8,1)$, pictured in
Figure~\ref{fig-2}.

\begin{figure}
\caption{The Kloosterman path $K_{17}(8,1)$}\label{fig-2}
\includegraphics[height=6cm]{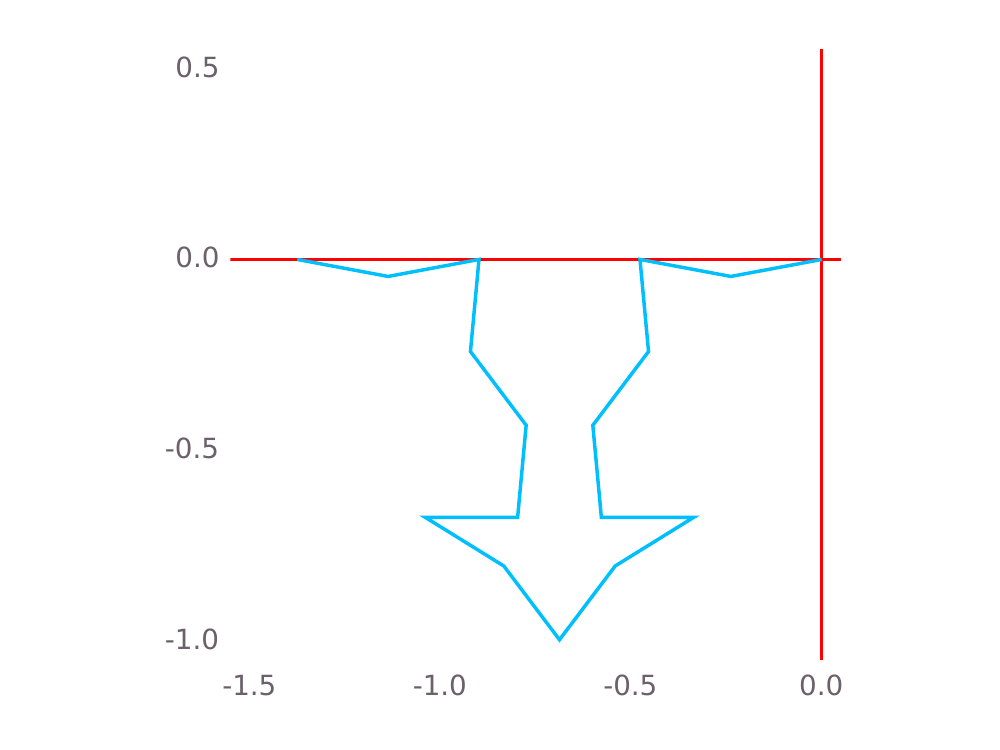}
\end{figure}

Despite these numbers, we can prove:

\begin{proposition}
  For all $p$ large enough, and all $(a,b)\in\Fpt\times\Fpt$, we have
  $f\notin\spt$.
\end{proposition}

\begin{proof}[Sketch of proof]
  By the Weyl criterion, for any fixed $k\geq 1$ and any tuple
  $(b_1,\ldots,b_k)$ of non-zero integers, the random variables
$$
a\mapsto
\Bigl(\frac{h}{p},\hypk(a,b_1;p),\ldots,\hypk(a,b_k;p)\Bigr)\in
\Rr/\Zz\times \Rr^k
$$
on $\Fp$ (with uniform probability measure) converge in law as
$p\to+\infty$ to independent random variables $(X_0,\ldots,X_{k})$
where $X_0$ is uniformly distributed in $\Rr/\Zz$ and
$(X_1,\ldots,X_k)$ are independent Sato-Tate random variables.  Using
the discrete Fourier expansion of $\hypkm(a-h,b;p)$, it follows that,
for any fixed $(a,b)$, the random variables
$$
h\mapsto \Bigl(\frac{h}{p},\hypkm(a-h,b;p)\Bigr)\in \Rr/\Zz\times \Cc
$$
on $\{0,\ldots,p-1\}$ (with uniform probability measure) converge in
law to $(X_0,\kpath(1/2))$ where $X_0$ is independent of
$\kpath(1/2)$. Moreover, the convergence is uniform in terms of
$(a,b)$.%%\footnote{TODO: explain this better.}
\par
Therefore, the random variable
$$
h\mapsto \frac{\sin(\pi h/p)}{\pi h/p} \widetilde{f}(h)
$$
converges in law to
$$
Y=2\frac{\sin(\pi X_0)}{\pi X_0}\Bigl(\cos(\pi X_0)\Reel(\kpath(1/2))+
\sin(\pi X_0)\Imag(\kpath(1/2))\Bigr).
$$
Since $X_0$ and $\kpath(1/2)$ are independent and the real part of
$\kpath(1/2)$ is between $-1$ and $1$, we have (say)
\begin{align*}
\proba(|Y|>2)&\geq \proba(|\Imag(K(1/2))|\geq 10\text{ and }
|X_0-1/4|\leq 1/10)\\
&=
\proba(|\Imag(K(1/2))|\geq 10)\proba(|X_0-1/4|\leq 1/10)>0
\end{align*}
since we showed in~\cite[Prop. 4.1]{ks} that $\Imag(K(1/2))$ can take
arbitrarily large values with positive probability. Hence, for all $p$
large enough, there exists $h$ such that
$$
\Bigl|\frac{\sin(\pi h/p)}{\pi h/p} \widetilde{f}(h)\Bigr|>2.
$$
% \begin{multline*}
% h\mapsto \frac{\sin(\pi h/p)}{\pi h/p} \widetilde{f}(h)\\=
% \frac{\sin(\pi h/p)}{\pi h/p} \Bigl(\cos(\pi h/p)\hypk_2(a-h,b;p)+
% 2\sin(\pi h/p)\Imag(\hypkm(a-h,b;p))\Bigr)
% \end{multline*}
% as a random variable on $\{0,\ldots,p-1\}$ with uniform probability
% measure.
\end{proof}

\end{example}

\begin{example}\label{ex-klpath3}
  Third-time lucky: the next variant of Kloosterman paths will always
  be realized in $\spt$. We now insert two pauses of duration $1/(2p)$
  at the beginning and end of the path. Thus $n=p+1$, $t_0=0$ and
  $t_{p+1}=1$, while $t_i=(i-\demi)/p$ for $1\leq i\leq p$; moreover
  $z_i$ is given by
$$
z_0=0,\quad\quad z_{p+1}=\hypk_2(a,b;p),
$$
and
$$
z_i=\frac{1}{\sqrt{p}} \sum_{1\leq x\leq
  i-1}e\Bigl(\frac{ax+b\bar{x}}{p}\Bigr)
$$
for $1\leq i\leq p$.
\par
Since the $t_i$'s are not all equal, the
formula~(\ref{eq-sums-polygonal}) does not apply, but we derive
from~(\ref{eq-fourier-polygonal}) that
\begin{align*}
  \widehat{f}(h)
  &= 
    -\frac{1}{2i\pi h} \hypk_2(a,b;p)+ 
    \frac{1}{2i\pi h}
    \frac{\sin(\pi h/p)}{\pi h/p}
    \frac{1}{\sqrt{p}}\sum_{x=1}^{p-1}e\Bigl(\frac{ax+b\bar{x}}{p}\Bigr)
    e\Bigl(-h\frac{x-\demi+\demi}{p}\Bigr)
  \\
  &=-\frac{1}{2i\pi h} \hypk_2(a,b;p)+ \frac{1}{2i\pi h}
    \frac{\sin(\pi h/p)}{\pi h/p}
    \hypk_2(a-h,b;p)
\end{align*}
for all $h\not=0$.  By the Weil bound for Kloosterman sums, we
conclude that $f\in\spt$.

As a consequence of the symmetry properties discussed in
Section~\ref{sec-structure}, all paths obtained by applying these
symmetries to these modified Kloosterman paths $f$ also belong to
$\spt$, and therefore can be approximated arbitrarily closely (in the
sense of Proposition~\ref{pr-arith}) by (actual!) Kloosterman paths.
This is quite remarkable, for instance because (at least if $p$ is
large enough) neither $-f$ nor $\bar{f}$ is associated to a
Kloosterman path (indeed, the pauses show that this would have to be
of the same type as $f$ for a Kloosterman path modulo the same prime
$p$, and comparing Fourier coefficients, one would need to have either
$-\hypk_2(a-h,b;p)=\hypk_2(c-h,d;p)$ for all $h$ or
$\hypk_2(a+h,b;p)=\hypk_2(c-d,d;p)$ for all $h$; both can be excluded
by elementary considerations concerning the Kloosterman sheaf).
% \par
% Of course, similar remarks apply to other examples below.
\end{example}

\begin{example}
  We proved in~\cite[Th. 1.3]{ks} that the random Fourier series
  $\kpath$ is also the limit of the processes $B_p$ of partial sums of
  Birch sums
$$
B(a;p)=\frac{1}{\sqrt{p}}\sum_{0\leq x\leq
  p-1}e\Bigl(\frac{ax+x^3}{p}\Bigr)
$$
where $a\in\Fp$ is taken uniformly at random. It is then natural to
consider these polygonal Birch paths and to ask whether they belong to
the support of $\kpath$. As defined, there is a trivial obstruction:
the path $t\mapsto B_p(a)(t)$ does not belong to $\fct{F}_0$, because
of the initial summand $1/\sqrt{p}$ for $x=0$.
\par
We can alter the path minimally by splitting the summand $1/\sqrt{p}$
in two summands $1/(2\sqrt{p})$ at the beginning and end of the
path. The resulting function, which we denote $f$, belongs to
$\fct{F}_0$. This means that we consider the polygonal path with
$n=p+1$, $t_i=(i-\demi)/p$ for $1\leq i\leq p$, and with $z_i$ defined
by
$$
z_0=0,\quad\quad z_{p+1}=B(a;p),
$$
and
$$
z_i=\frac{1}{2\sqrt{p}}+\frac{1}{\sqrt{p}} \sum_{1\leq j\leq
  i-1}e\Bigl(\frac{aj+j^3}{p}\Bigr)
$$
for $1\leq i\leq p$.
\par
As in the previous example, from~(\ref{eq-fourier-polygonal}) we get
\begin{multline*}
\widehat{f}(h)= -\frac{1}{2i\pi h} B(a;p)+ 
\frac{1}{2i\pi h}
\Bigl\{
\frac{1}{2\sqrt{p}}e\Bigl(-\frac{h}{4p}\Bigr)
\frac{\sin(\pi h/(2p))}{\pi h/(2p)}\\
+
\frac{\sin(\pi h/p)}{\pi h/p}
\frac{1}{\sqrt{p}}\sum_{x=1}^{p-1}e\Bigl(\frac{(a-h)x+x^3}{p}\Bigr)
+
\frac{1}{2\sqrt{p}}e\Bigl(\frac{h}{4p}\Bigr)
\frac{\sin(\pi h/(2p))}{\pi h/(2p)}
\Bigr\}.
\end{multline*}
The inner expression is equal to
$$
\frac{1}{\sqrt{p}}\frac{\sin(\pi h/(2p))}{\pi h/(2p)}
\cos\Bigl(\frac{\pi h}{2p}\Bigr) + \frac{\sin(\pi h/p)}{\pi
  h/p}\Bigl(B(a-h;p)-\frac{1}{\sqrt{p}}\Bigr)= \frac{\sin(\pi
  h/p)}{\pi h/p}B(a-h;p).
$$
By the Weil bound for Birch sums, we conclude that $f\in\spt$.
\end{example}

\begin{example}
  Let $p$ be a prime and $\chi$ a non-trivial Dirichlet character
  modulo $p$. We consider the polygonal paths interpolating the
  partial sums of the multiplicative character sum
$$
\frac{1}{\sqrt{p}}\sum_{1\leq x\leq p-1}\chi(x).
$$
\par
Let $f$ be the parameterized path where we insert pauses of duration
$1/(2p)$ at the beginning and at the end. Note that $f(1)=0$ by
orthogonality of characters. As in the previous computations, we get
\begin{align*}
  \widehat{f}(h)
  &= 
    \frac{1}{2i\pi h}
    \frac{\sin(\pi h/p)}{\pi h/p}
    \frac{1}{\sqrt{p}}\sum_{x=1}^{p-1}\chi(x)
    e\Bigl(-h\frac{x-\demi+\demi}{p}\Bigr)
  \\
  &=\frac{1}{2i\pi h}
    \frac{\sin(\pi h/p)}{\pi h/p}
    \chi(-1)\tau(\chi)\overline{\chi(h)},
\end{align*}
where
$$
\tau(\chi)=\frac{1}{\sqrt{p}}\sum_{1\leq x\leq
  p-1}\chi(x)e\Bigl(\frac{x}{p}\Bigr)
$$
is the normalized Gauss sum associated to $\chi$ (note that
$\chi(h)=0$ if $p\mid h$). Since $|\tau(\chi)|=1$, it follows that
$|\widehat{f}(h)|\leq 1$.  However, the Fourier coefficients are only
in $i\Rr$ (i.e., $f\in\fct{F}_0$) if $p\equiv 1\mods{4}$ and $\chi$ is
a real character.  In other words, Kloosterman sums can perfectly
mimic the character sums associated to the Legendre symbol modulo
such primes. (Note that in this case, the function $f$ is
real-valued).
\par
Note that character sums as above have been very extensively studied
from many points of view, because of their importance in many problems
of analytic number theory, for instance in the theory of Dirichlet
$L$-functions. We refer for instance to the
works~\cite{granville-sound, bg, bggk} of Bober, Goldmakher,
Granville, Koukoulopoulos and Soundararajan (in various
combinations). It should be possible (and interesting) to study the
support of the limiting distribution of these character paths, but
this will be very different from $\spt$. Indeed, one can expect
(see~\cite{bggk}) that the support in this case would be continuous
functions with totally multiplicative Fourier coefficients. For
instance, one can expect that $0$ does not belong to the support in
that case.
\end{example}

\begin{example}
  More generally, consider a prime $p$ and the polygonal path $f$
  associated to the partial sums of any exponential sum
$$
\frac{1}{\sqrt{p}}
\sum_{1\leq x\leq p}\chi(g_1(x))e\Bigl(\frac{g_2(x)}{p}\Bigr),
$$
where $\chi$ is a Dirichlet character modulo $p$, and $g_1$ and $g_2$
are polynomials in $\Zz[X]$ (with $g_2$ non-constant).  After suitable
tweaks, the Fourier coefficients become
$$
\widetilde{f}(h)= \frac{\sin(\pi h/p)}{\pi h/p}\frac{1}{\sqrt{p}}
\sum_{1\leq x\leq p}\chi(g_1(x))e\Bigl(\frac{g_2(x)-xh}{p}\Bigr).
$$
Assuming $f\in \fct{F}_0$, and under suitable restrictions, we may
expect that $f\in\spt$ only if the geometric monodromy group of the
Fourier transform of the rank $1$ sheaf with trace function the
summand
$$
x\mapsto \chi(g_1(x))e\Bigl(\frac{g_2(x)}{p}\Bigr)
$$
has rank $r$ at most $2$ (otherwise, Deligne's equidistribution
theorem will lead in most cases to the existence of $h$ such that
$0\leq h\leq p-1$ and $|\widetilde{f}(h)|\geq (r-1/2)>2$).
\end{example}

\begin{example}
  A natural question is whether $\spt$ contains a space-filling
  curve. Among the classical examples of such curves, the Hilbert
  curve~\cite[Ch. 2]{sagan} has a sequence of quite simple polygonal
  approximations $f_n$ for $n\geq 1$ that belong to $\fct{F}_0$
  (see~\cite[p. 14]{sagan}). We have in Figure~\ref{fig-hilbert} the
  plots of the second, third and fourth such approximations (note that
  there are many backtrackings, so this is a case where the plot
  doesn't give a clear idea of the path followed).

\begin{figure}
\caption{Approximations of the Hilbert function}\label{fig-hilbert}
\includegraphics[width=6cm]{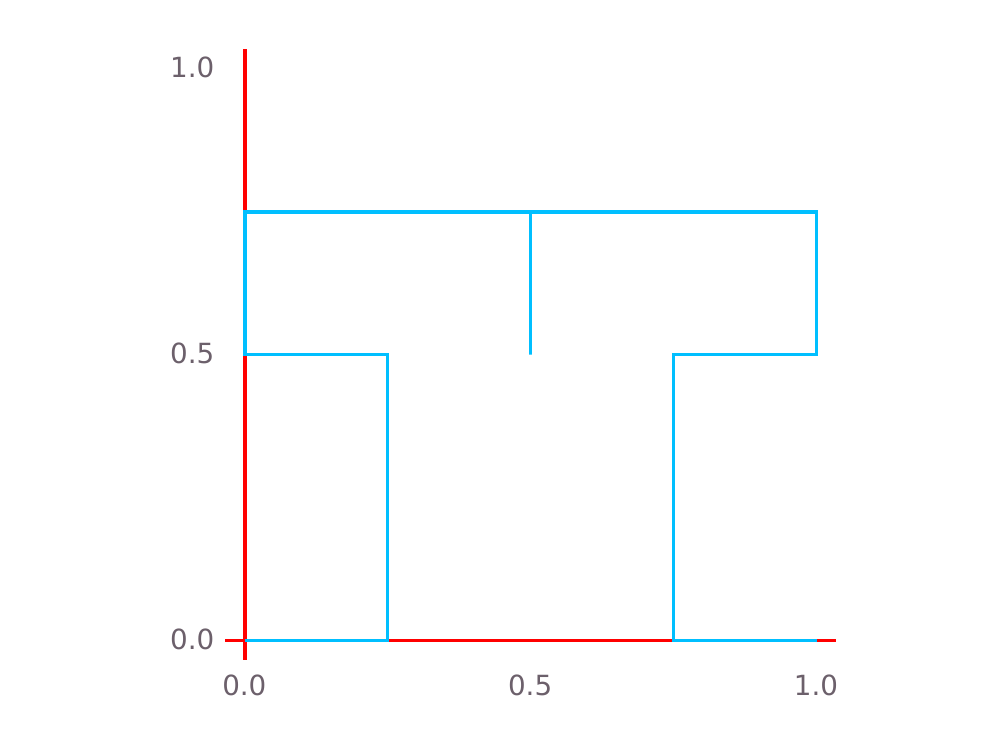}
\includegraphics[width=6cm]{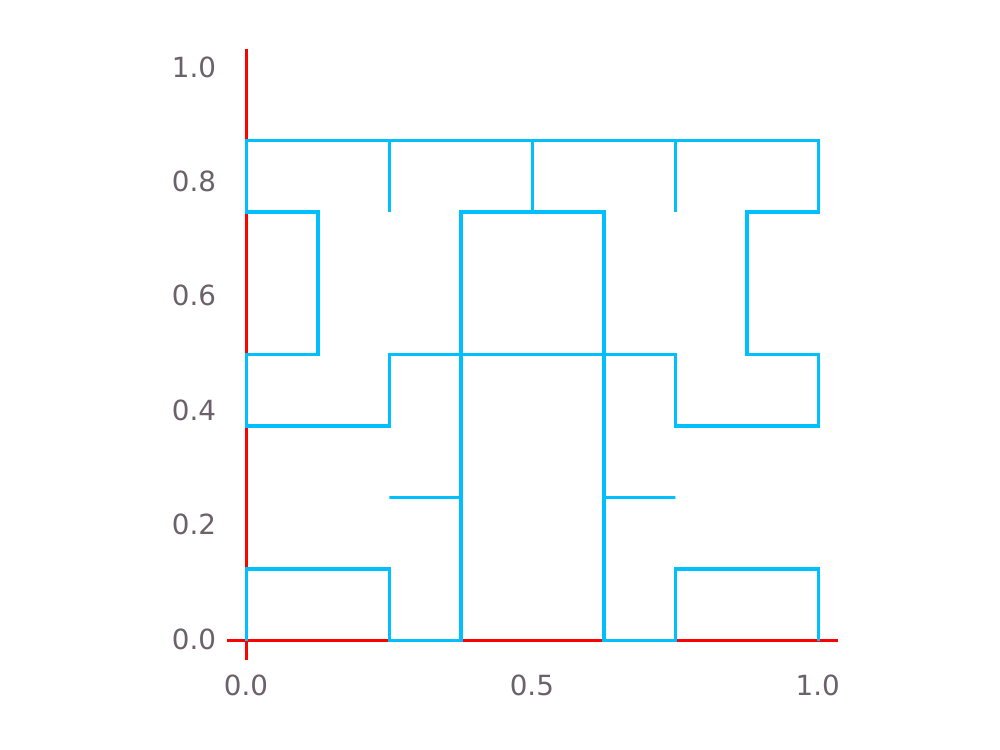}
\includegraphics[width=6cm]{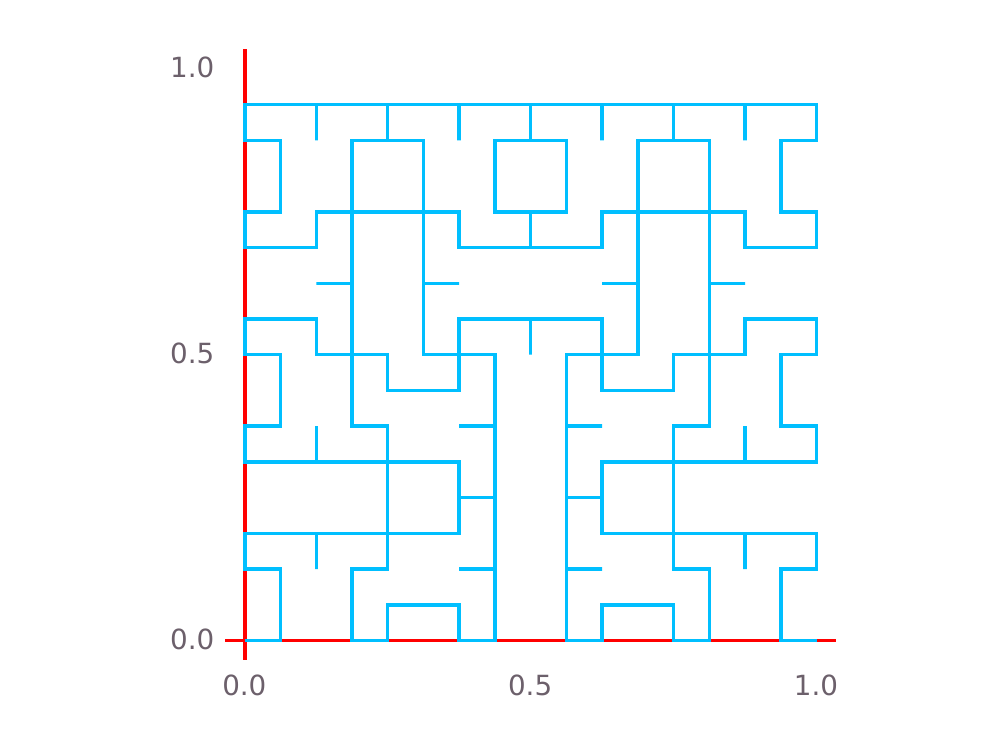}
\end{figure}

The function $f_n$ is a polygonal path composed of $4^n$ segments of
length $2^{-n}$. One checks that the Fourier coefficients are given by
$$
\widetilde{f}_n(h)= \frac{1}{2^n}
\sum_{j=0}^{4^n-1}i^{\delta_n(j)}e\Bigl(-\frac{h(j+1/2)}{4^n}\Bigr),
$$
for $h\not=0$, where the exponents $\delta_n(j)$ (in $\Zz/4\Zz$) are
determined inductively by
$$
\delta_1(0)=1,\quad \delta_1(1)=\delta_1(2)=0,\quad \delta_1(3)=3,
$$
and 
$$
\delta_{n+1}(4j)=1-\delta_n(j),\quad
\delta_{n+1}(4j+1)=\delta_{n+1}(4j+2)=\delta_n(j),\quad
\delta_{n+1}(4j+3)=3-\delta_n(j)
$$
for $n\geq 1$ and $0\leq j\leq 4^n-1$. The requirement for $f_n$ to
belong to $\spt$ is satisfied when these sums exhibit precisely the
analogue of the Weil bounds for $1\leq h\leq 4^n-1$. This may or may
not happen, and it turns out (numerically) that the first three
approximations are in $\spt$, but not the fourth.

%% Blocks
%% 0 --> 1003
%% 1 --> 0112
%% 2 --> 3221
%% 3 --> 2330

%% ** n=2

%% 1,0,0,3

%% ** n=3

%% 0,1,1,2   
%% 1,0,0,3  
%% 1,0,0,3   
%% 2,3,3,0

%% ** n=4

%% 1,0,0,3   
%% 0,1,1,2  
%% 0,1,1,2   
%% 3,2,2,1

%% 0,1,1,2   
%% 1,0,0,3  
%% 1,0,0,3   
%% 2,3,3,0

%% 0,1,1,2   
%% 1,0,0,3
%% 1,0,0,3   
%% 2,3,3,0

%% 3,2,2,1   
%% 2,3,3,0  
%% 2,3,3,0   
%% 1,0,0,3

\end{example}

\section{Changing the parameterization}\label{sec-fourier}

When we display the picture of a Kloosterman path, we are really only
seeing the \emph{image} of the corresponding function from $[0,1]$ to
$\Cc$. Although it is not really an arithmetic question anymore, it
seems fairly natural to ask which subsets of $\Cc$ are really going to
appear. This may be interpreted in different ways: (1) given a
function $f$ in $\fct{F}_0$, but not in $\spt$, when does there exist
a change of variable $\varphi\colon [0,1]\to [0,1]$ such that
$f\circ \varphi$ belongs to $\spt$?  (2) given a compact subset
$X\subset \Cc$, when does there exist an element $f\in\spt$ such that
$X=f([0,1])$?
\par
\emph{A priori}, these questions might be quite different. However, we
first show that the second essentially reduces to the
first. Precisely, we have a topological characterization of images of
functions in $\fct{F}_0$.

\begin{proposition}\label{pr-h-m-symmetric}
  Let $X\subset \Cc$ be a compact subset. The following conditions are
  equivalent:
\begin{enumerate}
\item There exists $f\in\fct{F}_0$
  such that $X$ is the image of $f$.
\item We have $0\in X$, there exists a real number $\alpha$ such that
  $X$ is symmetric with respect to the line $\Reel(z)=\alpha$, and
  there exists a continuous function $f\in C([0,1])$ such that
  $X=f([0,1])$.
\item We have $0\in X$, there exists a real number $\alpha$ such that
  $X$ is symmetric with respect to the line $\Reel(z)=\alpha$, and $X$
  is connected and locally connected.
\end{enumerate}
\end{proposition}

\begin{proof}
  It is immediate that (1) implies (2). Conversely, assume that (2)
  holds and let $f$ be a continuous function such that $f([0,1])=X$.
  Let $r\colon \Cc\to\Cc$ be the symmetry along the line
  $\Reel(z)=\alpha$, so that $X=r(X)$.  By assumption, there exist
  $s_0\in [0,1]$ and $s_1\in [0,1]$ be such that $f(s_0)=0$ and
  $f(s_1)=r(0)=2\alpha$. Up to replacing $f$ by $t\mapsto f(1-t)$, we
  may assume that $s_0\leq s_1$.
% \par
% If $\alpha=0$, we simply define
% $$
% g(t)= \begin{cases}
%   f(2t)&\text{ if } 0\leq t\leq 1/2\\
%   -\overline{f(2(1-t))}&\text{ if } 1/2\leq t\leq 1.
% \end{cases}
% $$
% This belongs to $\fct{F}_0$ (because $f(1)=0$) and has the same image
% as $f$.
\par
Let $T$ be the set of all $t\in [0,1]$ such that $t\geq s_0$ and
$\Reel(f(t))=\alpha$. This set is closed and it is non-empty (because
the image of the continuous real-valued function $\Reel(f)$ contains
$0=f(s_0)$ and $r(0)=2\alpha=f(s_1)$ by assumption, and
$s_1\geq s_0$). Let $t_0=\max T$ and $Y=f([0,t_0])\cup
r(f([0,t_0]))$. We claim that $X=Y$. Indeed, suppose some $x\in X$ is
not in $Y$. Then we also have $r(x)\notin Y$. Hence we can write
$x=f(t_1)$ with $t_1>t_0$ and $r(x)=f(t_2)$ with $t_2>t_0$.  Then
% $\Reel(f(t_1))\not=\alpha$ (otherwise $t\in T$, so $t$ would be
% $\leq t_0$), and
$$
\alpha=\frac{1}{2}(\Reel(f(t_2))+\Reel(f(t_1)),
$$
so $\alpha$ is in the interval between $\Reel(f(t_1))$ and
$\Reel(f(t_2))$. By continuity, there exists $s$ between $t_1$ and
$t_2$ with $\Reel(f(s))=\alpha$, contradicting the maximality of
$t_0$.
\par
Now define
% Let $T$ be the set of all
%   $t\in [0,1]$ such $X=f([0,t])\cup r(f([0,t]))$. This is a non-empty
%   set (since $1\in T$) and it is closed by an elementary compactness
%   argument. Let $t_0=\min T\in [0,1]$. We define $g$ by
$$
g(t)=
\begin{cases}
  f(s_0(1-8t))&\text{ if } 0\leq t\leq 1/8\\
  f(2s_0(t-1/8))  & \text{ if } 1/8\leq t\leq 1/4\\
f(s_0+4(t_0-s_0)(t-1/4))  & \text{ if } 1/4\leq t\leq 1/2
\end{cases}
$$
and $g(t)=r(g(1-t))=2\alpha-\overline{g(1-t)}$ if $1/2<t\leq 1$ (in
other words, $g(t)$ covers the path of $f$ from $0=f(s_0)$ to $f(0)$
for $t\in [0,1/8]$, then covers it backwards from $t=1/8$ to $t=1/4$,
then follows the path over $[1/4,1/2]$ from $0$ to $f(t_0)$, and then
proceeds by reflection).
\par
We have $g(0)=0$ and $g$ is continuous (because
$\Reel(g(1/2))=\Reel(f(t_0))=\alpha$), hence $g\in \fct{F}_0$ by
construction.  The image of $g$ is contained in $X$; it contains
$f([0,t_0])$ and its reflection, so its image is $X$. This proves (1)
for the set $X$.
\par
To prove that (2) and (3) are equivalent, we simply need to invoke the
Hahn-Mazurkiewicz Theorem (see, e.g.,~\cite[Th. 6.8]{sagan}
or~\cite[TA, III, p. 272, th. 1]{bourbaki}): a non-empty compact
subset $X\subset \Cc$ is the image of a continuous function
$f\colon [0,1]\to\Cc$ if and only if $X$ is connected and locally
connected.
\end{proof}

Because of this proposition, it is natural to concentrate on the
change of variable problem. Here a subtlety is whether we wish to have
an invertible reparameterization or not: if
$\varphi\colon [0,1]\to [0,1]$ is merely surjective, the image of
$f\circ\varphi$ is the same as that of $f$. However, we consider here
only transformations $\varphi$ that are homeomorphisms. In fact, let
us say that an increasing homeomorphism $\varphi$ of $[0,1]$ such that
$\varphi(1-t)=1-\varphi(t)$ is a \emph{symmetric} homeomorphism. We
then have $f\circ\varphi\in\fct{F}_0$ for all $f\in\fct{F}_0$. The
question is: for a given $f\in\fct{F}_0$, does there exist a symmetric
homeomorphism $\varphi$ such that $f\circ\varphi\in\spt$?
\par
To prove our result for real-valued functions in
Proposition~\ref{pr-repar}, we will use a variant of a result of
Sahakian\footnote{\ Also spelled Saakjan, Saakian,
  Saakyan.}~\cite[Cor. 2]{saakjan}.
\par
Recall that the Faber-Schauder functions $\Lambda_{m,j}$ on $[0,1]$
are defined for $m\geq 0$ and $1\leq j\leq 2^m$ by the following
conditions:
\begin{itemize}
\item The support of $\Lambda_{m,j}$ is the dyadic interval 
$$
\Bigl[\frac{j-1}{2^m},\frac{j}{2^m}\Bigr],
$$
of length $2^{-m}$, 
\item We have $\Lambda_{m,j}((2j-1)2^{-m-1})=1$,
\item The function $\Lambda_{m,j}$ is affine on the two intervals
$$
\Bigl[\frac{j-1}{2^m},\frac{2j-1}{2^{m+1}}\Bigr],\quad\quad
\Bigl[\frac{2j-1}{2^{m+1}},\frac{j}{2^{m}}\Bigr].
$$
\end{itemize}
\par
Any continuous function $f$ on $[0,1]$ has a uniformly convergent
Faber-Schauder series expansion
$$
f(t)= \beta(0)+\beta(1)t+\sum_{m\geq
  0}\sum_{j=1}^{2^m}\beta(m,j)\Lambda_{m,j}(t),
$$
with coefficients
$$
\beta(0)=f(1),\quad\quad \beta(1)=f(1)-f(0),
$$
and
\begin{equation}\label{eq-fs}
  \beta(m,j)=f\Bigl(\frac{2j-1}{2^{m+1}}\Bigr)- \frac{1}{2}\Bigl(
  f\Bigl(\frac{j-1}{2^{m}}\Bigr)+ f\Bigl(\frac{j}{2^{m}}\Bigr) \Bigr)
\end{equation}
(see, e.g.,~\cite[Ch. VI]{kashin-saakyan} for these facts). The
function $f$ is $1$-periodic if and only if $\beta(1)=0$.

\begin{theorem}[Sahakian]\label{th-saakjan}
  Let $g\colon [0,1]\to \Rr$ be a \emph{real-valued} continuous
  function with $g(0)=0$. Let $\eps>0$ be any fixed positive real
  number. 
\par
\emph{(1)} There exists an increasing homeomorphism
$\varphi\colon [0,1]\to [0,1]$ such that the Fourier coefficients of
the function $u(g\circ \varphi)=g\circ \varphi-g(1)t$ satisfy
$$
|\widehat{u(g\circ \varphi)}(h)|\leq\frac{\eps}{|h|}
$$
for all $h\not=0$. 
\par
\emph{(2)} If the function $g$ satisfies %%either
$g(t)+g(1-t)=g(1)$ for all $t$,
%%, or $g(1-t)=g(t)$ for all $t$, 
then we may assume that $\varphi$ is symmetric.
\end{theorem}

We emphasize that the function $g$ is real-valued; it does not seem to
be known whether the statement (1) holds for a complex-valued function
$g$.  The issue in the proof in~\cite{saakjan} is the essential use of
the intermediate value theorem.\footnote{\ One might hope to extend
  the proof to any continuous function $f\colon [0,1]\to \Cc$
  satisfying the intermediate value property, in the sense that the
  image $f([s,t])$ of any interval $[s,t]\subset [0,1]$ contains the
  segment $[f(s),f(t)]$ (or equivalently such that $f([s,t])$ is
  always convex), but it is an open question of Mihalik and Wieczorek
  whether such functions exist that do not take values in a line in
  $\Cc$ (see the paper of Pach and Rogers~\cite{p-r} for the best
  known result in this direction.)}
  
\begin{proof}[Sketch of proof]
  Below, we will say that a continuous function $g\colon [0,1]\to\Cc$
  is $1$-periodic if $g(0)=g(1)$, which means that the periodic
  extension of $g$ to $\Rr$ is continuous.
\par
The result requires only very minor changes in Sahakian's argument,
which does not address exactly this type of uniform ``numerical''
bounds, but asymptotic statements like
$|\widehat{(g\circ \varphi)}(h)|=o(|h|^{-1})$ as $|h|\to+\infty$ when
$g$ is $1$-periodic.%% (these are stronger for large $h$, of course).
\par
For any continuous $1$-periodic function $f$ on $[0,1]$, extended to
$\Rr$ by periodicity, define
$$
\omega_f(\delta)=\sup_{0<\alpha\leq \delta}
\int_0^1|f(x+\alpha)+f(x-\alpha)-2f(x)|dx.
$$
A classical elementary argument (compare~\cite[II.4]{Z}) shows that
for a $1$-periodic function $f$, we have
\begin{equation}\label{eq-zygmund}
|\widehat{f}(h)|\leq \frac{1}{4}\omega_f\Bigl(\frac{1}{|h|}\Bigr)
\end{equation}
for all $h\not=0$. It is also elementary that there exists $C>0$ such
that
$$
\omega_{\Lambda_{m,j}}(\delta)\leq C\min(2^m\delta^2,2^{-m})
$$
for all $m$ and $j$.
\par
By~\cite[Lemma 1]{saakjan}, applied to the continuous real-valued
function $t\mapsto g(2\pi t)$ on $[0,2\pi]$, there exists a
homeomorphism $\varphi$ such that, for any $m\geq 0$, the coefficients
$\beta(m,j)$ of the Faber-Schauder expansion of $g\circ \varphi$
vanish for all but at most one index $j_m$, and moreover, we have
$$
|\beta(m,j_m)|<\frac{\eps}{C}.
$$
Note that the text of~\cite{saakjan} might suggest that the lemma is
stated for $1$-periodic functions, but the proof is in fact written
for arbitrary continuous functions (as it must, since it proceeds by
an inductive argument from $[0,1]$ to dyadic sub-intervals, and any
periodicity assumption in the construction would be lost after the
first induction step).
\par
Let $\gamma_m=\beta(m,j_m)$ and $\Phi_m=\Lambda_{m,j_m}$. Since
$g(1)=g(\varphi(1))$, we have the series expansion
$$
u(g\circ\varphi)(t)=(g\circ \varphi)(t)-g(1)t=\sum_{m\geq
  0}\gamma_m\Phi_{m}(t),
$$
uniformly for $t\in [0,1]$ and hence, using the subadditivity of
$f\mapsto \omega_f$, we get
$$
\omega_{u(g\circ \varphi)}(\delta)\leq \eps \sum_{m\geq
  0}\min(2^m\delta^2,2^{-m}) \leq 4\eps\delta.
$$
By~(\ref{eq-zygmund}), we get
$$
|(\widehat{u(g\circ\varphi)})(h)|\leq \frac{\eps}{|h|}.
$$
for $h\not=0$, which proves the first statement.
\par
Consider now the case when the condition $g(t)+g(1-t)=g(1)$ holds. We
then apply the previous argument (properly scaled) to the restriction
of $g$ to $[0,1/2]$, obtaining an increasing homeomorphism $\psi$ of
$[0,1/2]$ such that
\begin{equation}\label{eq-faber1}
  (g\circ \psi)(t)-2g(1/2)t=(g\circ\psi)(t)-g(1)t
  =\sum_{m\geq 1}\gamma_m \Phi_{m}(t)
\end{equation}
for $0\leq t\leq 1/2$ where $|\gamma_m|\leq \eps C^{-1}$ and $\Phi_m$
is a Faber-Schauder function associated to an interval of length
$2^{-m}$ of $[0,1/2]$.
\par
We define $\varphi\colon [0,1]\to [0,1]$ so that $\varphi$ coincides
with $\psi$ on $[0,1/2]$ and $\varphi(1-t)=1-\varphi(t)$ for
$0\leq t\leq 1/2$. Then $\varphi$ is a symmetric homeomorphism of
$[0,1]$. Because of the symmetry of $g$ and~(\ref{eq-faber1}), we have
for $1/2\leq t\leq 1$ the formula
\begin{align*}
  (g\circ \varphi)(t)=g(1)-g(\varphi(1-t))
  &=
    g(1)-(1-t)g(1)-\sum_{m\geq 1}\gamma_m\Phi_m(1-t)\\
  &=
    g(1)t-\sum_{m\geq 1}\gamma_m \Phi_{m}(1-t).
\end{align*}
Since the supports are disjoint, we can therefore write
% the definition of the
% Faber-Schauder functions and the formula~(\ref{eq-fs}) for the
% Faber-Schauder coefficients, we have the expansion
$$
u(g\circ \varphi)(t)= (g\circ \varphi)(t)-g(1)t= \sum_{m\geq
  1}\gamma_m \Phi_{m}(t)- \sum_{m\geq 1}\gamma_m \Phi_{m}(1-t)
$$
for \emph{all} $t\in [0,1]$. Now we evaluate the Fourier coefficients
as before.
% inspection of the proof of~\cite[Lemma 1]{saakjan} shows that in that
% case, the homeomorphism $\varphi$ can be taken to satisfy
% $\varphi(1-t)=1-\varphi(t)$. More precisely:
% \begin{itemize}
% \item Aat the initial step~\cite[p. 553]{saakjan}, one can take
%   $a_1^1=1/2$ because the symmetry of $g$ ensures that
%   $\demi(g(0)+g(1))=g(1/2)$, so Case 1) in loc. cit. holds, and
%   $a_1^1=1/2$ is acceptable.
% \item After determining the values $\varphi(j/2^m)=a_m^j$ for
%   $j\leq 2^{m-1}$ as in~\cite[p. 554]{saakjan}, one may prescribe
%   $\varphi(1-j/2^m)$ to satisfy the desired relation; the symmetry of
%   $g$ again shows that this assignment satisfies the conditions
%   spelled out in~\cite[p. 552]{saakjan}).
% \end{itemize}
%%The case when $g(1-t)=g(t)$ is similar.
\end{proof}

We can now prove Proposition~\ref{pr-repar}.

% \begin{proposition}\label{pr-repar}
%   Let $f\in\fct{F}_0$ such that $|f(1)|\leq 2$. Assume that either $f$
%   is real-valued, or that it takes values in $i\Rr$.
% % one of
% %   the following conditions holds:
% % \par
% % \emph{(1)} The function $f$ is real-valued.
% % \par
% % \emph{(2)} The function $f$ takes values in $i\Rr$.
% % \par
% % \emph{(3)} We have $\widehat{f}(h)=0$ for all $h<0$.
% % \par
% % \emph{(4)} We have $\widehat{f}(h)=0$ for all $h>0$.
% % \par
% Then there exists a symmetric increasing homeomorphism
% $\varphi\colon [0,1]\to [0,1]$
% % such that $\varphi(1-t)=1-\varphi(t)$ for all $t$, and
% such that $f\circ \varphi\in\spt$.
% \end{proposition}

\begin{proof}[Proof of Proposition~\ref{pr-repar}]
  Let $f$ be a real-valued function $f\in\fct{F}_0$ with
  $|f(1)|\leq 2$. Theorem~\ref{th-support} and
  Theorem~\ref{th-saakjan} (2) applied to $f$ (which satisfies
  $f(t)+f(1-t)=f(1)$ since it is real-valued) with
  $\eps=1/\pi$ imply the existence of the desired reparameterization.
% . Similarly, if $f$ takes values in $i\Rr$, we can apply
%   Theorem~\ref{th-saakjan} (3) to $\Imag(f)$ and $\eps=1/\pi$.
%   \par
% In Case (3), we observe that for any $f\in\fct{F}_0$, since the
% Fourier coefficients of $f$ belong to $i\Rr$, the continuous function
% $g=\Reel(f)$ has real Fourier series
% $$
% \sum_{h\geq 1} i(\widehat{f}(h)-\widehat{f}(-h))\sin(2\pi ht).
% $$
% If $\widehat{f}(h)=0$ for $h<0$, then this series becomes
% $$
% \sum_{h\geq 1} i\widehat{f}(h)\sin(2\pi ht).
% $$
% By Case (1), there exists a symmetric homeomorphism
% $\varphi\colon [0,1]\to [0,1]$ such that $g\circ \varphi\in\spt$. 
\end{proof}

% We conclude by showing that, in general, reparameterization of a
% function $f\in\fct{F}_0$ with $|f(1)|\leq 2$ is not possible.

% \begin{proposition}
%   Let $f\in\fct{F}_0$ be of the form $f=i\alpha \Lambda_{0,1}+f_1$
%   where $|\alpha|\geq 2$ and $f_1\in\fct{F}_0$ satisfies $f_1(1)=0$
%   and $\|f_1\|_{\infty}\leq 1/10$. For any symmetric homeomorphism
%   $\varphi$, we have $f\circ \varphi\notin \spt$.
% \end{proposition}

% \begin{proof}

% \end{proof}

%%%
%% f=\Lambda_{0,1}
%% \widehat{f}(h)=-(1-(-1)^h)\frac{1}{4\pi^2h^2}
%%%

\begin{remark}
% (1)   Let $f$ be any element of $\fct{F}_0$ with $f(1)=0$. Since the
%   Fourier coefficients of $f$ belong to $i\Rr$, the continuous
%   periodic function $g=\Reel(f)+\Imag(f)$ has real Fourier series
% $$
% i\sum_{h\geq 1} \Bigl((\widehat{f}(h)-\widehat{f}(-h))\sin(2\pi ht)+
% (\widehat{f}(h)+\widehat{f}(-h))\cos(2\pi ht)\Bigr).
% $$
% Hence there exists an increasing homeomorphism $\varphi$ of $[0,1]$
% such that $g\circ \varphi$ has real Fourier expansion
% $$
% \sum_{h\geq 1} \Bigl(a_h\sin(2\pi ht)+b_h\cos(2\pi ht)\Bigr),
% $$
% with
% $$
% |a_h|+|b_h|\leq \frac{1}{2\pi|h|}
% $$
% for all $h\geq 1$. However, since $\varphi$ is not necessarily
% symmetric, the Fourier coefficients of $f\circ\varphi$ are not
% necessarily purely imaginary, which means we can only conclude
% something like that
% $$
% |\Reel(\widehat{(f\circ\varphi)}(h))
% +\Imag(\widehat{(f\circ\varphi)}(h)|\leq \frac{1}{\pi|h|}
% $$
% for all $h\not=0$.  
% \par
(1) The prototypical statement of ``improvement'' of convergence of a
Fourier series by change of variable is the Bohr-P\'al Theorem (see,
e.g.,~\cite[Th. VII.10.18]{Z}), which gives for any $1$-periodic
continuous \emph{real-valued} function $f$ a homeomorphism $\varphi$
of $[0,1]$ such that the Fourier $f\circ \varphi$ converges uniformly
on $[0,1]$. The extension to complex-valued functions was obtained by
Kahane and Katznelson~\cite{kk}.
\par
(2) It seems that the problem of obtaining the bound
$\widehat{f\circ \varphi}(h)=O(|h|^{-1})$ for a complex-valued
$1$-periodic function $f\in C([0,1])$ is quite delicate. For instance,
let $W_2^{1/2}$ be the Banach space of integrable functions $f$ on
$[0,1]$ such that
$$
\sum_{h\in\Zz}|h||\widehat{g}(h)|^2<+\infty.
$$
Let $f_1$ be a real-valued $1$-periodic function in
$C([0,1])$. Lebedev~\cite[Th. 4]{lebedev} proves that if $f_1$ has the
property that, for any $f\in C([0,1])$ with real part $f_1$, there
exists an homeomorphism $\varphi$ such that both
$f_1\circ \varphi=\Reel(f)\circ \varphi$ and $\Imag(f)\circ \varphi$
belong to $W_2^{1/2}$, then $f_1$ is of bounded variation (and indeed,
the converse is true).
% \par
% (3) On the other hand, let $f\in\fct{F}_0$ be such that the imaginary
% part of $f$ is of bounded variation (for instance, is $C^1$). If
% $|f(1)|\leq 2$, then we can apply Proposition~\ref{pr-repar} to
% $\Reel(f)\in\fact{F}_0$ and obtain a symmetric homeomorphism $\varphi$
% such that $\Reel(f)\circ \varphi\in\spt$. Although $f\circ \varphi$
% might not belong to $\spt$, since the imaginary part of $f\circ
% \varphi$ is of bounded
\par
(3) Note that in any reparameterization $f\circ \varphi$ of
$f\in \fct{F}_0$ with $\varphi$ symmetric, the coefficient
$\beta(0,1)$ of the Faber-Schauder function $\Lambda_{0,1}$ is
unchanged: because $\varphi(1/2)=1/2$, it is
$$
\beta(0,1)=f\Bigl(\frac{1}{2}\Bigr)-\frac{1}{2}(f(0)+f(1))=\Imag(f(\demi)).
$$
In particular, one cannot hope to reparameterize all functions with
$f(1/2)\notin \Rr$ using information on the Faber-Schauder expansion
of $f\circ\varphi$ and individual estimates for each Faber-Schauder
function that is involved.
\end{remark}

\end{document}